\documentclass[11pt,a4paper]{article}

\usepackage[a4paper]{geometry}
\usepackage{amssymb}
\usepackage{amsthm}
\usepackage{graphicx, gastex}
\usepackage{amsmath}
\usepackage{latexsym}
\usepackage{amsfonts}
\usepackage{color}
\usepackage{vaucanson-g}
\usepackage{stmaryrd}
\usepackage{extarrows}
\usepackage{cite}

\DeclareSymbolFont{rsfscript}{OMS}{rsfs}{m}{n}
\DeclareSymbolFontAlphabet{\mathrsfs}{rsfscript}

\DeclareMathOperator{\Core}{Core}
\DeclareMathOperator{\Sch}{Sch}

\DeclareMathOperator{\St}{Stab}
\DeclareSymbolFont{rsfscript}{OMS}{rsfs}{m}{n}
\newtheorem{theorem}{Theorem}
\newtheorem{prop}{Proposition}

\newtheorem{lemma}{Lemma}
\newtheorem{cor}{Corollary}

\newcommand{\rf}{\rightarrow}

\newcommand{\la}{\langle}
\newcommand{\ra}{\rangle}
\newcommand{\ul}{\underline}
\newcommand{\wt}{\widetilde}

\newcommand{\inv}{^{-1}}

\newcommand{\oo}{\overline}
\newcommand{\To}{Tor}

\def\vlongrightarrow{\relbar\joinrel\longrightarrow}
\def\vvlongrightarrow{\relbar\joinrel\vlongrightarrow}
\def\vvvlongrightarrow{\relbar\joinrel\vvlongrightarrow}
\def\vvvvlongrightarrow{\relbar\joinrel\vvvlongrightarrow}
\def\vvvvvlongrightarrow{\relbar\joinrel\vvvvlongrightarrow}
\def\vvvvvvlongrightarrow{\relbar\joinrel\vvvvvlongrightarrow}

\def\vlongmapright#1{\smash{\mathop{\vvlongrightarrow}\limits^{#1}}}
\def\vvlongmapright#1{\smash{\mathop{\vvvlongrightarrow}\limits^{#1}}}

\newcommand{\longfr}[2]{\smash{\stackrel{\text{\tiny{$#1|#2$}}}{\vlongrightarrow}}}
\newcommand{\vlongfr}[2]{\smash{\stackrel{\text{\tiny{$#1|#2$}\,}}{\vvlongrightarrow}}}
\newcommand{\vvlongfr}[2]{\smash{\stackrel{\text{\tiny{$#1|#2$}\,}}{\vvvlongrightarrow}}}
\newcommand{\vvvlongfr}[2]{\smash{\stackrel{\text{\tiny{$#1|#2$}\,}}{\vvvvlongrightarrow}}}
\newcommand{\exlongfr}[2]{\smash{\stackrel{\text{\tiny{$#1|#2$}\,}}{\vvvvvvlongrightarrow}}}
\newcommand{\mapright}[1]{\smash{\stackrel{\text{\tiny{$#1$}}}{\vlongrightarrow}}}
\newcommand{\path}[1]{\smash{\stackrel{\text{\tiny{$#1$}}}{\longrightarrow}}}

\title{A geometric approach to (semi)-groups defined by automata via dual transducers} 
\author{Daniele D'Angeli\\
        Institut f\"{u}r mathematische Strukturtheorie (Math C)\\
Technische Universit\"{a}t Graz\\
Steyrergasse 30, 8010 Graz, Austria.\\
        \texttt {dangeli@math.tugraz.at}
        \and
       	Emanuele Rodaro \\
        Department of Mathematics, University of Porto\\
        Rua do Campo Alegre, 687, Porto, 4169-007, Portugal.\\				
        \texttt{emanuele.rodaro@fc.up.pt}
        }

\date{\today}

\begin{document}
\maketitle

\begin{abstract}
We give a geometric approach to groups defined by automata via the notion of enriched dual of an inverse transducer. Using this geometric correspondence we first provide some finiteness results, then we consider groups generated by the dual of Cayley type of machines. Lastly, we address the problem of the study of the action of these groups in the boundary. We show that examples of groups having essentially free actions without critical points lie in the class of groups defined by the transducers whose enriched dual generate a torsion-free semigroup. Finally, we provide necessary and sufficient conditions to have finite Schreier graphs on the boundary yielding to the decidability of the algorithmic problem of checking the existence of Schreier graphs on the boundary whose cardinalities are upper bounded by some fixed integer.
\end{abstract}

\section{Introduction}\label{sec:intro}
This paper frames into the setting of the study of the properties of graphs and groups generated by finite automata. This theory became very popular after the introduction of the (first) Grigorchuk's group as the first example of a group with intermediate growth, i.e super-polynomial and sub-exponential (see, for example, \cite{DynSubgroup}). The class of groups generated by automata, or invertible transducers, also contains groups with special and interesting properties, among these, in the last years people have highlighted a very strong and surprising connection with complex dynamics and dynamical system. It is worth mentioning here the seminal works of V. Nekrashevych (see \cite{volo}), that contributed, for example, in clarifying the correspondence between expanding complex maps (and the associated Julia sets) and contracting self-similar groups (and their Schreier graphs). Schreier graphs, naturally appear in this context: they correspond to the stabilizers of words in $A^{\ast} \sqcup A^{\omega}$ and can be depicted as orbital graphs of the action of the generators of the group on $A^n$ and $A^{\omega}$. Since the action of such groups preserves the uniform Bernoulli measure on $A^{\omega}$, one can study the dynamical system given by the action of $\mathcal{G}(\mathrsfs{A})$ on the boundary of the tree and the relative action on the space of its subgroups stabilizers \cite{DynSubgroup}.
Given $\mathcal{G}(\mathrsfs{A})$, the problem of classifying (topologically and isometrically) its Schreier graphs on the boundary $A^{\omega}$ is still open. Partial results were obtained in \cite{basilica} for the action of the Basilica group and in \cite{BDN} for groups generated by bounded automata (see also \cite{alf}). Here two kinds of actions are particularly interesting: the totally non free actions and the essentially free actions. The latter corresponds to groups whose stabilizers on the boundary are almost all trivial, the former is given by groups whose stabilizers on the boundary are almost all distinct (see \cite{Vershik}). In the case of essentially free actions, it is an open question weather it is possible to find examples of non residually abelian groups acting without critical points, i.e. elements in the boundary whose stabilizer does not coincide with the stabilizer of any its neighbor. In this work we address these questions by studying the geometry of an enriched version $(\partial \mathrsfs{A})^{-}$ of the dual automaton $\partial \mathrsfs{A}$, obtained by exchanging the state space and the input/output set of the automaton $\mathrsfs{A}$ and adding arrows labelled by formal inverses of the input/output alphabet in order to equip $\partial \mathrsfs{A}$ with a structure of inverse transducer. The study of such automaton is the key tool of the paper, since it encodes very well the action of the group $\mathcal{G}(\mathrsfs{A})$ and can offer an alternative approach for attacking the problems that we have presented above. In particular, in Section \ref{sec: dynamics}, we show that examples of groups with all trivial stabilizers in the boundary can be found only in the class of transducers whose enriched dual generates a torsion-free semigroup. The properties captured by the enriched dual automaton are used to study three different types of the so called Cayley machines, regarded as dual automata of groups. Two machines generate finite groups, while the dual of the usual Cayley machine generates in general a group, containing the free semigroup, and for which we provide a recursive way to build its presentation.
A useful tool that we introduce in this paper is the components growth $\chi_{\mathrsfs{A}}(n)$ defined as the size of the smallest component of the $n-$th power of the dual automaton of $\mathrsfs{A}$. This allows us to give a decidable algorithm to determine the existence of a finite Schreier graph on the boundary.

The paper is organized as follows: in Section \ref{sec: preliminaries} and Section \ref{sec: inverse} we introduce the notation and the notion of inverse transducer that is crucial for our analysis. Section \ref{sec: dual} introduces the enriched dual transducer and highlights its geometric relationship with the original transducer by showing its connection with the Schreier and the Cayley graphs of the group defined by the original transducer. Then we pass to study the finiteness conditions for (semi-)groups generated by transducers by introducing the notion of supesymmetric automata. The power of our method is demonstrated by the fact that we can easily recover some important known results (see, for example, Corollary \ref{cor:piccantino} in Section \ref{sec: finiteness}). In Section \ref{sec: cayley} we study transducers modeling Cayley graphs of finite groups. The last section contains the results relative to the study of the Schreier graphs on the boundary of the tree.

\section{Preliminaries}\label{sec: preliminaries}
A word $w$ over a finite \emph{alphabet} $A$ is a tuple $w=(w_{1},\ldots, w_{n})$ of element of $A$ which is more often represented as a string $w=w_{1}\ldots w_{n}$, and for convenience we will use both the notations freely. In the sequel $A$ denotes a finite set, called \emph{alphabet}, $A^{*}$ ($A^{+}$) is the free monoid with identity $1$ (semigroup) on $A$. By $A^{\le n}$ ($A^{\ge n}$, $A^{n}$) we denote the set of words of length less or equal (greater or equal, equal) to $n$. With $A^{\omega}$ we denote the set of right infinite words in $A$, we use the vector notation, and for an element $\ul{u}=u_{1}u_{2}\ldots u_{i}\ldots\in A^{\omega}$ the prefix of length $k>0$ is denoted by $\ul{u}[k]=u_{1}u_{2}\ldots u_{k}$, while the factor $u_{i}\ldots u_{j}$ is denoted by $\ul{u}[i,j]$.
In this paper we deal mostly with automata and transducers from a geometric point of view, this means that we deal with paths, connected components, and, in general properties of their underling graphs. Therefore, to have a common notation both for transducers and for automata, we present them as directed labelled graphs. Hence, in our context a directed graph (for short \emph{digraph}) is a graph in the sense of Serre (see for instance \cite{Serre}). Thus, it is a tuple $(V,E,\iota,\tau)$, where $V$ is the set of vertices, $E$ is the set of edges, and $\iota, \tau$ are functions from $E$ into $V$ giving the initial and terminal vertices, respectively. We may depict an edge $e\in E$ as $e=q\mapright{}q'$ where $q=\iota(e), q'=\tau(e)$. The respective labelled structure is an $A$-labelled directed graph (for short an $A$-digraph) which is a tuple $\Gamma=(V,E,A,\iota,\tau,\mu)$ where $(V,E,\iota,\tau)$ is a digraph and $\mu:E\rightarrow A$ is the \emph{labeling} map. In this case we depict $e\in E$ with $q=\iota(e), q'=\tau(e)$, $\mu(e)=a$ as $e=q\mapright{a}q'$. A path is a sequence of edges $p=e_{1},\ldots,e_{k}$ such that $\tau(e_{i})=\iota(e_{i+1})$ for $i=1,\ldots, k-1$, and we say that the origin of $p$ is $\iota(p)=\iota(e_{1})$ and the terminal vertex is $\tau(p)=\tau(e_{k})$. The \emph{label} of the path $p$ is the word $\mu(p)=\mu(e_{1})\ldots \mu(e_{k})$, and we graphically represent this path as $p=v\vlongmapright{\mu(p)}v'$.
When we fix a vertex $v\in V$ (a \emph{base point}), the pair $(\Gamma, v)$ can be seen as a language recognizer ($A$-automaton), whose language recognized is the set:
$$
L(\Gamma,v)=\{\mu(p):p\mbox{ is a path in }\Gamma\mbox{ with }\iota(p)=\tau(p)=v\}
$$
When we pinpoint the vertex $v$, we implicitly assume that the underlying $A$-digraph of $(\Gamma,v)$ is the connected component of $\Gamma$ containing $v$. For a graph $\Gamma$ with set of vertices $V(\Gamma)$, for any $v\in V(\Gamma)$ we denote $\|(\Gamma, v)\|$ the cardinality of the set of vertices of a connected component of $\Gamma$ containing $v$, and we put $\|\Gamma\|=\max_{v\in V}\{\|(\Gamma, v)\|\}$. An important operation between automata is the \emph{product} of automata. Given two automata $(\Gamma_{1},v_{1})$, $(\Gamma_{2},v_{2})$ with $\Gamma_{1}=(V_{1},E_{1},A,\iota_{1},\tau_{1},\mu_{1})$, $\Gamma_{2}=(V_{2},E_{2},A,\iota_{2},\tau_{2},\mu_{2})$ the product is the automaton is given by
$$
(\Gamma_{1},v_{1})\times(\Gamma_{2},v_{2})=\left((V_{1}\times V_{2}, D,A,\iota_{1}\times \iota_{2}, \tau_{1}\times\tau_{2}, \mu_{1}\times\mu_{2}), (v_{1},v_{2})\right)
$$
where $D\subseteq E_{1}\times E_{2}$ is the set of pair of edges $(e_{1},e_{2})$ such that $\mu_{1}(e_{1})=\mu_{2}(e_{2})$. It is a standard fact that the product of two automata recognizes the intersection of the two languages recognized, i.e.
$$
L((\Gamma_{1},v_{1})\times(\Gamma_{2},v_{2}))=L(\Gamma_{1},v_{1})\cap L(\Gamma_{2},v_{2})
$$
A morphism $\psi: \Gamma\rightarrow \Gamma'$ between the two $A$-digraphs $\Gamma=(V,E,A,\iota,\tau,\mu)$ and $\Gamma'=(V',E',A,\iota',\tau',\mu')$ consists of a pair $(\psi_{V},\psi_{E})$ of maps $\psi_{V}: V\rf V'$ and $\psi_{E}:E\rf E'$ preserving the $A$-digraph structure, i.e. $\iota'(\psi_{E}(e))=\psi_{V}(\iota(e))$, $\tau'(\psi_{E}(e))=\psi_{V}(\tau(e))$, and it is labeling preserving $\mu(e)=\mu'(\psi_{E}(e))$. A morphism of two automata $\psi: (\Gamma,v)\rightarrow (\Gamma',v')$ is just a morphism $\psi$ of the two underlying $A$-digraphs preserving the base points. The $A$-labelled graph $\Gamma$ is called \emph{complete} (\emph{deterministic}) if for each vertex $v\in V$ and $a\in A$ there is (at most) an edge $e\in E$ with $\iota(e)=v$ and $\mu(e)=a$.
\\
More commonly in literature a deterministic and complete $A$-labelled graph $\Gamma$ with a finite number of vertices is referred as \emph{semiautomaton} \cite{HowieAuto} and it can be equivalently described by a $3$-tuple $\mathcal{A}=(Q,A,\delta)$ where $Q$ is a finite set of states, $A$ is a finite alphabet, $\delta:Q\times A\rf Q$ is the \emph{transition function}. Fixing a base-point $q\in Q$ the language recognized by the automaton (DFA) $(\mathcal{A}, q)$ is the set $L(\mathcal{A},q)=\{u\in A^{*}: \delta(q,u)=q\}$.
This notation is compatible with the one presented above since this language is the same as the language of the underlying $A$-digraph of $\mathcal{A}$ with base point $q$. Another way of seeing the map $\delta$ is as an action $Q\overset{\cdot}{\curvearrowleft} A^{*}$ of $A^{*}$ on $Q$ defined inductively by the formula $q\cdot (a_{1}\ldots a_{n})=\delta(q, a_{1})\cdot  (a_{2}\ldots a_{n})$ for any $q\in Q$, $a_{1}\ldots a_{n}\in A^{*}$. The semiautomaton $\mathcal{A}$ is called \emph{reversible} whenever this action is a permutation, or equivalently the maps $\delta(-,a):Q\rf Q$ are permutations for any $a\in A$.
\\
In this paper we clearly consider alphabetical transducers with the same input and output alphabet, for further details on the general theory of automata and transducers we refer the reader to \cite{Eil, HowieAuto}. A \emph{finite state Mealy automaton}, shortly a transducer, is a $4$-tuple $\mathrsfs{A} = (Q,A,\delta,\lambda)$ where $(Q,A,\delta)$ is a semiautomaton, while $\lambda:Q\times A\rf A$ is called the \emph{output function}. This function defines an action $Q\overset{\circ}{\curvearrowright} A^{*}$ of $Q$ on $A^{*}$ defined inductively by
$$
q\circ (a_{1}\ldots a_{n})=\lambda(q,a_{1})\left((q\cdot a_{1})\circ (a_{2}\ldots a_{n})\right)
$$
Usually, in the theory of groups (semigroup) defined by automata, people are interested in the action $Q\overset{\circ}{\curvearrowright} A^{*}$ of the group (semigroup) on the rooted tree $A^*\cup A^{\omega}$.
However, in this paper we are dealing with the dual of a transducer in which the role is reversed. For this reason, it is convenient to deal with these actions in a symmetric way by extending them to $Q^{*}$ in the natural way. Thus, the pair $(Q^{*}\overset{\cdot}{\curvearrowleft} A^{*}$, $Q^{*}\overset{\circ}{\curvearrowright} A^{*})$ is called the associated \emph{coupled-actions} of the transducer $\mathrsfs{A}$, and henceforth we will write $\mathrsfs{A}=(Q,A,\cdot,\circ)$.
\\
From the geometrical point the transducer $\mathrsfs{A}$ can be visualized as an $A\times A$-labelled digraph $(Q,E,A\times A,\iota, \tau, \mu)$ with edges of the form $q\mapright{a|b}q'$ whenever $q\cdot a=q'$ and $q\circ a=b$, and we will make no distinction between the transducer and the digraph notation. Thus, given a transducer $\mathrsfs{A}=(Q,A,\cdot,\circ)$ we implicitly assume that it also has a structure $\mathrsfs{A}=(Q,E,A\times A,\iota, \tau, \mu)$ of $A\times A$-labelled digraph. Considering just the input or the output labeling, we may define the \emph{input automaton} $\mathrsfs{A}_{\mathcal{I}}=(Q,E,A,\iota, \tau, \mu_{1})$ where $\mu_{1}(e)=a$ whenever $e=q\mapright{a|b}q'$ is an edge of $\mathrsfs{A}$, and dually the \emph{output automaton} is $\mathrsfs{A}_{\mathcal{O}}=(Q,E,A,\iota, \tau, \mu_{2})$ where $\mu_{2}(e)=b$ whenever $e=q\mapright{a|b}q'$ is an edge of $\mathrsfs{A}$. Sometimes, we may write $L(\mathrsfs{A},q)=L(\mathrsfs{A}_{\mathcal{I}},q)$.
\\
The \emph{product} of the two machines $\mathrsfs{A} = (Q,E,A\times A,\iota, \tau, \mu)$, $\mathrsfs{B} = (T,D,A\times A,\iota', \tau', \mu')$ is the machine $\mathrsfs{A}\mathrsfs{B} = (Q\times T,F,A\times A,\oo{\iota},\oo{\tau},\oo{\mu})$ whose edges are given by $(q,q')\mapright{a|b}(p,p')$ whenever $q\mapright{a|c}p$ is an edge in $E$ and  $q'\mapright{c|b}p'$ is an edge in $D$. The $k$-th power of the machine $\mathrsfs{A}$ is defined inductively by $\mathrsfs{A}^{k}=(\mathrsfs{A}^{k-1})\mathrsfs{A}$, and we put $\mathrsfs{A}^{k}_{\mathcal{I}}=(\mathrsfs{A}^{k})_{\mathcal{I}}$. For the study of the dynamic in the boundary we are interest in limits of these powers. Consider the sequence $\{\mathrsfs{A}^{k}_{\mathcal{I}}\}$ of $A\times A$-graphs with the morphisms $\phi_{i,i-1}$ mapping the edge $(q_{1},\ldots q_{k})\mapright{a}(p_{1},\ldots, p_{k})$ into the edge $(q_{1},\ldots q_{k-1})\mapright{a}(p_{1},\ldots, p_{k-1})$. Thus, $(\{\mathrsfs{A}^{k}_{\mathcal{I}}\}_{k\ge 1}, \phi_{i,j})$ forms an inverse system. The following proposition is a standard fact of the projective limit.
\begin{prop}\label{prop: projective}
Let $\varprojlim \{\mathrsfs{A}^{k}_{\mathcal{I}}\}_{k\ge 1}=\mathrsfs{A}_{\mathcal{I}}^{\infty}$ be the inverse limit of the inverse system $(\{\mathrsfs{A}^{k}_{\mathcal{I}}\}_{k\ge 1}, \phi_{i,j})$, and let $\varphi_{k}:\mathrsfs{A}_{\mathcal{I}}^{\infty}\rf \mathrsfs{A}^{k}_{\mathcal{I}}$ be the natural maps. Then, for any $\ul{q}\in Q^{\omega}$ we have:
$$
L(\mathrsfs{A}_{\mathcal{I}}^{\infty}, \ul{q})=\bigcap_{k>0} L\left (\mathrsfs{A}^{k}_{\mathcal{I}}, \varphi_{k}(\ul{q})\right)
$$
with $ \varphi_{k}(\ul{q})=\ul{q}[k]$.
\end{prop}
From the algebraic point of view the action $Q^{*}\overset{\circ}{\curvearrowright} A^{*}$ gives rise to a semigroup $\mathcal{S}(\mathrsfs{A})$ generated by the endomorphisms $\mathrsfs{A}_{q}$, $q\in Q$, of the rooted tree identified with $A^{*}$ defined by $\mathrsfs{A}_{q}(u)=q\circ u$, $u\in A^{*}$. For $q_{1},\ldots, q_{m}\in Q$ we may use the shorter notation $\mathrsfs{A}_{q_{1}\ldots q_{m}}=\mathrsfs{A}_{q_{1}}\ldots \mathrsfs{A}_{q_{m}}$. An important role in group theory is played by groups defined by invertible transducers, for more details we refer the reader to \cite{volo}. A transducer $\mathrsfs{A} = (Q,A,\cdot,\circ)$ is called \emph{invertible} whenever the map $\lambda(q,\circ):A\rf A$ is a permutation. In this case all the maps $\mathrsfs{A}_{q}$, $q\in Q$, are automorphisms of the rooted regular tree identified with $A^{*}$, and the group generated by these automorphisms is denoted by $\mathcal{G}(\mathrsfs{A})$. Henceforth a generator $\mathrsfs{A}_{q}$ of $\mathcal{G}(\mathrsfs{A})$ is identified with the element $q\in Q$, and its inverse with the formal inverse $q^{-1}\in Q^{-1}=\{q^{-1}:q\in Q\}$. The action of $\mathcal{G}(\mathrsfs{A})$ naturally extends on the \textit{boundary} $A^{\omega}$ of the tree. We are interested in faithful actions hence, throughout the paper we assume all the transducers to have this property. Notice that, since $|Q|<\infty$, the action of $\mathcal{G}(\mathrsfs{A})$ on $A^{\omega}$ cannot be transitive and so it decomposes into uncountably many orbits. Given $v\in A^{\ast}\sqcup A^{\omega}$, the stabilizer $\St_{\mathcal{G}(\mathrsfs{A})}(v)=\{g\in \mathcal{G}(\mathrsfs{A}) \ : \ g(v)=v\}$ is a subgroup of $\mathcal{G}(\mathrsfs{A})$. The stabilizer of the $n-$th level is the normal subgroup $\St_{\mathcal{G}(\mathrsfs{A})}(n)=\cap_{v\in A^n} \St_{\mathcal{G}(\mathrsfs{A})}(v)$.
In many cases one tries to determine the stabilizers of the elements in $A^{\omega}$.
The stabilizers are strongly related to the Schreier graphs: given $v\in A^{\ast}\sqcup A^{\omega}$ and its stabilizer $\St_{\mathcal{G}(\mathrsfs{A})}(v)$, the Schreier graph $\Sch(\St_{\mathcal{G}(\mathrsfs{A})}(v), Q\cup Q^{-1})$ corresponds to the orbital graph of $v$ under the action of the generators $Q\cup Q^{-1}$. For an invertible transducer $\mathrsfs{A} = (Q,E,A\times A,\iota, \tau, \mu)$ we define the inverse (transducer) $\mathrsfs{A}^{-1} = (Q^{-1},E,A\times A,\iota, \tau, \mu')$, and there is an edge $q^{-1}\mapright{a|b}p^{-1}$ in $\mathrsfs{A}^{-1}$ whenever $q\mapright{b|a}p$ is an edge in $\mathrsfs{A}$.
\\
Another two important classes of transducers that we consider throughout the paper are the \emph{reversible} and \emph{bireversible} machines. A transducer $\mathrsfs{A}$ is called reversible whenever $\mathrsfs{A}_{\mathcal{I}}$ is a reversible semiautomaton, and it is called bireversible if in addition also $\mathrsfs{A}_{\mathcal{O}}$ is a reversible semiatomaton, hence in this case $\mathrsfs{A}$ must be necessarily invertible. A reversible invertible transducer will be also called $RI$-transducer. The following proposition shows that reversibility is preserved under the product of machines.
\begin{prop}\label{prop: product reversible}
Let $\mathrsfs{A}$, $\mathrsfs{B}$ be two transducers. If $\mathrsfs{A}$ and $\mathrsfs{B}$ are reversible (bireversible), then $\mathrsfs{A}\mathrsfs{B}$ is reversible (bireversible). Moreover, if $\mathrsfs{A}$ satisfies the property that for any $a\in A$ there is an edge $q\mapright{b|a}q'$, then $\mathrsfs{A}\mathrsfs{B}$ is reversible if and only if both $\mathrsfs{A}$ and $\mathrsfs{B}$ are reversible.
\end{prop}
\begin{proof}
Suppose that $\mathrsfs{A}\mathrsfs{B}$ is not reversible, and so there are two distinct edges
$$
(q,q')\mapright{a|b}(p,p'),\quad (s,s')\mapright{a|c}(p,p')
$$
in $\mathrsfs{A}\mathrsfs{B}$ with $(q,q')\neq (s,s')$. Therefore, let $q\mapright{a|d}p$, $s\mapright{a|e}p$, and $q'\mapright{d|b}p'$, $s'\mapright{e|c}p'$ be the corresponding edges of $\mathrsfs{A}$, $\mathrsfs{B}$, respectively. If $q\neq s$, then $\mathrsfs{A}$ is not reversible, and we are done. Thus, we can assume $q=s$, from which we get $d=e$. Since $(q,q')\neq (s,s')$, then $q'\neq s'$, from which we get $s'\mapright{e|c}p'$, $q'\mapright{e|b}p'$ are two different edges in $\mathrsfs{B}$, whence $\mathrsfs{B}$ is not reversible. If both $\mathrsfs{A}$ and $\mathrsfs{B}$ are bireversible, then they are invertible and their inverses $\mathrsfs{A}^{-1}$, $\mathrsfs{B}^{-1}$ are necessarily reversible. Thus, since $\left(\mathrsfs{A}\mathrsfs{B}\right)^{-1}=\mathrsfs{B}^{-1}\mathrsfs{A}^{-1}$, then by the previous result $\left(\mathrsfs{A}\mathrsfs{B}\right)^{-1}$ is reversible, whence $\mathrsfs{A}\mathrsfs{B}$ is bireversible.
\\
Let us prove the last statement, we need to prove the only if part. Thus, assume that $\mathrsfs{A}\mathrsfs{B}$ is reversible. If both $\mathrsfs{A}$ and $\mathrsfs{B}$ are reversible, then we are done. Thus, suppose that $\mathrsfs{A}$ is not reversible and $\mathrsfs{B}$ is reversible, and let $q\mapright{a|b}p$, $s\mapright{a|c}p$ be two edges of $\mathrsfs{A}$ with $q\neq s$. Since $\mathrsfs{B}$ is reversible, there are edges $q'\mapright{b|d}p'$, $s'\mapright{c|e}p'$ in $\mathrsfs{B}$. Hence $(q,q')\mapright{a|d}(p,p')$ and $(s,s')\mapright{a|e}(p,p')$ are two edges of $\mathrsfs{A}\mathrsfs{B}$ with $(q,q')\neq (s,s')$, whence $\mathrsfs{A}\mathrsfs{B}$ is not reversible. On the other hand, suppose that $\mathrsfs{A}$ is reversible and $\mathrsfs{B}$ is not reversible, and let $q'\mapright{c|d}p'$, $s'\mapright{c|e}p'$ be two edges of $\mathrsfs{B}$ with $q'\neq s'$, and by the condition on $\mathrsfs{A}$, let $q\mapright{a|c}p$ be an edge in $\mathrsfs{A}$. Hence, in $\mathrsfs{A}\mathrsfs{B}$ there are two edges $(q,q')\mapright{a|d}(p,p')$, $(q,s')\mapright{a|e}(p,p')$, a contradiction.
\end{proof}

\section{Inverse graphs, inverse transducers and free groups}\label{sec: inverse}

Let $A$ be a finite alphabet and let $\tilde{A}=A\cup A\inv$ be the
\emph{involutive alphabet} where $A\inv$ is the set of \emph{formal}
inverses of $A$. The operator $\inv:A\rightarrow A\inv: a\mapsto
a\inv$ is extended to an involution on the free monoid $\wt{A}^*$ through
$$
1\inv = 1, \;\; (a\inv)\inv=a, \;\; (uv)\inv=v\inv u\inv\;\;\; (a\in
A;\;u,v\in\wt{A}^*).
$$
Let $\sim$ be the congruence on $\wt{A}^*$ generated by the relation
$\{(aa\inv,1)\mid a \in \wt{A} \}$. The
quotient $F_A= \wt{A}^*/\sim$ is the \emph{free group} on $A$, and throughout the paper
$\sigma:\wt{A}^* \to F_A$ denotes the canonical homomorphism. The set of all reduced words on
$\wt{A}^*$, may be compactly written as
$$
R_A=\tilde{A}^*\setminus\bigcup_{a\in\tilde{A}}\tilde{A}^*aa\inv\tilde{A}^*
$$
For each $u\in\tilde{A}^*$, $\overline{u}\in R_A$ is the (unique) reduced word $\sim$-equivalent to $u$. We write also $\overline{u\sigma}=\overline{u}$. As usual, we often identify the elements of $F_A$ with their reduced representatives. For $B\subseteq \wt{A}^*$, $\oo{B}$ denotes the set of reduced words of $B$. An $\wt{A}$-digraph $\Gamma$ is called \emph{involutive} if whenever $p\mapright{a}q$ is and edge of $\Gamma$, so is $q\mapright{a^{-1}}p$; the graph $\Gamma$ is called \emph{inverse} if in addition $\Gamma$ is deterministic. When we depict an inverse graph we can draw just one between the two edges $p\mapright{a}q$, $q\mapright{a^{-1}}p$, $a\in\wt{ A}$, this corresponds to chose an orientation $E^{+}$ on the set of edges $E$. For inverse graphs there is an important property which relates languages to morphisms:
\begin{prop}\label{prop: immersion lang hom}
Given two inverse graphs $\Gamma_{1},\Gamma_{2}$, and two vertices $q_{1},q_{2}$ belonging to $\Gamma_{1},\Gamma_{2}$, respectively. Then $L(\Gamma_{1},v_{1})\subseteq L(\Gamma_{2},v_{2})$ if and only if there is a morphism $\varphi:(\Gamma_{1},v_{1})\rightarrow (\Gamma_{2},v_{2})$. Furthermore, $(\Gamma_{1},v_{1})$ is the minimal inverse automaton (up to isomorphisms) recognizing $L(\Gamma_{1},v_{1})$.
\end{prop}
\begin{proof}
The proof belongs to the folklore, see for instance \cite{BarSil}.
\end{proof}
Let $\Gamma=(V,E,A,\iota,\tau,\mu)$ be an inverse graph, when we fix a base point $v$ of $\Gamma$, the pair $(\Gamma,v)$ is often referred as an \emph{inverse automaton}. There is an important connection between inverse automata and subgroups of $F_{A}$, we recall the basic facts and we refer the reader to \cite{BarSil,KaMi02} for more details.
A path $p\mapright{z}q$ in $\Gamma$ is called \emph{reduced} if $\oo{z}=z$. It is not difficult to see that if $p\vvlongmapright{uvv^{-1}w}q$ is a path in $\Gamma$, then $p\mapright{uw}q$ is also a path. Therefore, if there is a path connecting two vertices, then there is also a reduced path connecting them, hence $\oo{L(\Gamma, v)}\subseteq L(\Gamma, v)$. We can consider a smaller inverse subautomaton with this property, the \emph{core} of $(\Gamma,v)$ is the induced inverse subautomaton $\Core(\Gamma,v)$ of $(\Gamma,v)$ containing all the reduced paths $p$ with $\iota(p)=\tau(p)=v$. From the language point of view this operation does not erase reduced words:
$$
\oo{L(\Gamma,v)}\subseteq L(\Core(\Gamma,v))\subseteq L(\Gamma,v)
$$
in particular $\oo{L(\Gamma,v)}=\oo{ L(\Core(\Gamma,v))}$. The connection between subgroups of $F_{A}$ and inverse graphs (automata) can be summarized in the following theorem.
\begin{theorem}\label{theo: inverse-groups}
Let $\Gamma=(V,E,A,\iota,\tau,\mu)$ be an inverse graph, $v\in V$, and let $Y$ be a spanning tree of $(\Gamma,v)$ (a connected subtree of $\Gamma$ with the same set of vertices of $(\Gamma,v)$). For each vertex $t\in (\Gamma,v)$ there is a unique path $p_{t}$ in $Y$ connecting $v$ to $t$. Fix an orientation $E^{+}$, and let $T^{+}\subseteq E^{+}$ be the set of edges lying outside $Y$. Then $H=\oo{L(\Gamma,v)}$ is a (free) subgroup of $F_{A}$ generated by:
$$
\{\oo{\mu(p_{s})\mu(e)\mu(p_{t})^{-1}}: e=s\vlongmapright{\mu(e)}t\in T^{+}\}
$$
Furthermore, if $\Sch(H,A)$ is the Schreier graph of $H$, and $(\Sch(H,A),H)$ is the associated inverse automata, then $\Core(\Gamma,v)=\Core(\Sch(H,A),H)$. Conversely, if $H$ is a subgroup of $F_{A}$, then $(\Sch(H,A),H)$ is an inverse automaton such that $\sigma^{-1}(H)=L(\Sch(H,A),H)$, in particular $H=\oo{L(\Sch(H,A),H)}$.
\end{theorem}
By this theorem for each inverse automaton $(\Gamma,v)$ we can associate with it a unique subgroup.
For a finitely generated group $H$, $\Core(\Sch(H,A),H)$ can be build via the so called \emph{Stallings construction}. This construction involves an operation, called \emph{Stallings foldings} which is used to transform an involutive graph into an inverse graph. If there are two edges $p\mapright{a}q$, $p\mapright{a}q'$ in an involutive graph $\Delta$ with $q\neq q'$ and $a\in\wt{A}$, then the \emph{folding} is performed by identifying these two edges, as well as the two respective inverse edges, obtaining the inverse graph $\oo{\Delta}$.
If the subgroup $H$ is generated by a finite set $X\subseteq R_A$ of reduced words, the \emph{Stallings automaton} $\mathcal{S}(X)$ is obtained considering the involutive graph $\mathcal{F}(X)$ consisting of all the paths $v\mapright{u}v$, $u\in X$, and then applying all the possible foldings to $\mathcal{F}(X)$ getting $\mathcal{S}(X)=(\oo{\mathcal{F}(X)},v)$. Note that this construction does not depend of the generating set $X$, whence we can write $\mathcal{S}(H)=\mathcal{S}(X)$.
\\
While inverse automata on the alphabet $A$ essentially represent subgroups of $F_{A}$, there is an important tool recently introduced by Silva in \cite{SilvaVirtual} which is a compact way to deal with maps on $F_{A}$: the class of \emph{inverse transducers}. Although in \cite{SilvaVirtual} they are introduced in a more general form, in our context we restrict them to transducers whose associated $\wt{A}\times \wt{A}$-digraphs are involutive, where the involution is given by $(a,b)\mapsto (a^{-1},b^{-1})$. Note that, by \cite[Proposition 3.1]{SilvaVirtual}, given an inverse transducer $\mathrsfs{A} = (Q,\wt{A},\cdot,\circ)$, for any $q\in Q$ the map $\mathrsfs{A}_{q}:\wt{A}^{*}\rightarrow \wt{A}^{*}$ induces a map $\wt{\mathrsfs{A}}_{q}:F_{A}\rightarrow F_{A}$ given by $\wt{\mathrsfs{A}}_{q}(\sigma(u))=\sigma(\mathrsfs{A}_{q}(u))$.
\\
A central, although evident, fact for this paper is given by the following lemma which shows that we can always enrich a reversible transducer with a structure of inverse transducer.
\begin{lemma}\label{lemma: enriching}
Let $\mathrsfs{A}=(Q,A,\cdot,\circ)$ be a reversible (bireversible) transducer, then it can be extended to a reversible (bireversible) transducer $\mathrsfs{A}^{-}=(Q,\wt{A},\cdot,\circ)$ by adding to each edge $q\mapright{a|b}p$ of $\mathrsfs{A}$, the edge $p\vlongfr{a^{-1}}{b^{-1}}q$.
\end{lemma}
\begin{proof}
Since $\mathrsfs{A}$ is reversible we have that $q\mapright{a|b}p$ and $q'\mapright{a|c}p$ implies $q=q'$. This condition ensures that $\mathrsfs{A}^{-}$ is a transducer. The fact that $\mathrsfs{A}^{-}$ is still reversible follows from the determinism of $\mathrsfs{A}$. Indeed, if by absurd we assume that  $q\vvlongfr{a^{-1}}{b^{-1}}p$, $q'\vvlongfr{a^{-1}}{c^{-1}}p$ are edges of $\mathrsfs{A}^{-}$ for some $a,b,c\in A$ and $q\neq q'$, then $p\mapright{a|b}q$, $p\mapright{a|c}q'$ are also edges of $\mathrsfs{A}$, a contradiction. The bireversibility of  $\mathrsfs{A}^{-}$  in case  $\mathrsfs{A}$ is bireversible arises from the the bireversibility of $\mathrsfs{A}$ itself, the details are left to the reader.
\end{proof}
Note that the property of being invertible is not preserved on the alphabet $\wt{A}$ in the passage from $\mathrsfs{A}$ to its associated inverse transducer $\mathrsfs{A}^{-}$. However, if $\mathrsfs{A}$ is bireversible this property is preserved. This operator also commutes with the product of transducers.
\begin{prop}\label{prop: inverse of products}
Let $\mathrsfs{A}$, $\mathrsfs{B}$ be two reversible transducers, then:
$$
\left(\mathrsfs{A}\mathrsfs{B}\right)^{-}=\mathrsfs{A}^{-}\mathrsfs{B}^{-}
$$
In particular, there is a morphism from $\mathcal{S}(\mathrsfs{A}^{-})$ onto $\mathcal{S}(\mathrsfs{A})$.
\end{prop}
\begin{proof}
This is an easy consequence of the fact that $(q,q')\vvlongfr{a^{-1}}{c^{-1}}(p,p')$ is an edge in $(\mathrsfs{A}\mathrsfs{B})^{-}$ if and only if $q\vvlongfr{a^{-1}}{b^{-1}}p$ in an edge in $\mathrsfs{A}^{-}$ and  $q'\vvlongfr{b^{-1}}{c^{-1}}p'$ is an edge in $\mathrsfs{B}^{-}$, which is equivalent to the fact that $p\mapright{a|b}q$, $p'\mapright{b|c}q'$ are edges in $\mathrsfs{A}$, $\mathrsfs{B}$, respectively.
\end{proof}

\section{A geometric perspective via the enriched dual}\label{sec: dual}
In this section we fix a transducer $\mathrsfs{A} = (Q,A,\cdot,\circ)$. The \emph{dual} is the (well defined) transducer $\partial\mathrsfs{A} = (A,Q,\circ,\cdot)$. It is a quite useful tool already used in \cite{Pic12,StVoVo2011, VoVo2007,VoVo2010}. It is straightforward to check that $\partial(\partial \mathrsfs{A})=\mathrsfs{A}$, and $\partial\mathrsfs{A}$ can be visually obtained by the correspondence
$$
p\longfr{a}{b}q\:\:\Longleftrightarrow\:\: a\longfr{p}{q}b
$$
The coupled-actions associated to $\mathrsfs{A}$ $(Q^{*}\overset{\cdot}{\curvearrowleft} A^{*},Q^{*}\overset{\circ}{\curvearrowright} A^{*})$ can be compactly described by the following two recursive equations:
$$
(q_{1}\ldots q_{n-1}q_{n})\cdot (a_{1}\ldots a_{k})=\left(\left(q_{1}\ldots q_{n-1}\right)\cdot (q_{n}\circ (a_{1}\ldots a_{k}))\right) q_{n}\cdot (a_{1}\ldots a_{k})
$$
$$
(q_{1}\ldots q_{n-1}q_{n})\circ (a_{1}\ldots a_{k})=(q_{1}\ldots q_{n-1})\circ (q_{n}\circ (a_{1}\ldots a_{k}))
$$
Hence, for $\partial\mathrsfs{A}$ the associated coupled-actions is given by $(A^{*}\overset{\circ}{\curvearrowleft} Q^{*},A^{*}\overset{\cdot}{\curvearrowright} Q^{*})$.
\\
The following proposition sums up some relationships between a transducer and its dual, some of them have already been observed in \cite{VoVo2007,VoVo2010}.
\begin{prop}\label{prop: dual prop}
Let $\mathrsfs{A}$ be a transducer, then:
\begin{enumerate}
 \item[i)] $\mathrsfs{A}$ is invertible if and only if $\partial\mathrsfs{A}$ is reversible;
 \item[ii)] $\mathrsfs{A}$ is a $RI$-transducer if and only if $\partial\mathrsfs{A}$ is a $RI$-transducer.
 \item[iii)] $\mathrsfs{A}$ is bireversible if and only if $\partial\mathrsfs{A}$ is bireversible;
\end{enumerate}
\end{prop}
\begin{proof}
Follows from the definitions.
\end{proof}
In case $\mathrsfs{A}$ is invertible, we can extend $\mathrsfs{A}$ to the disjoint union $\mathrsfs{A}\sqcup \mathrsfs{A}^{-1}$ in the obvious way. This extension is clearly reflected on the coupled-action $(\wt{Q}^{*}\overset{\cdot}{\curvearrowleft} A^{*},\wt{Q}^{*}\overset{\circ}{\curvearrowright} A^{*})$. Note that the action of group $\mathcal{G}(\mathrsfs{A})$ on $A^{*}$ is the same as $\wt{Q}^{*}\overset{\circ}{\curvearrowright}A^{*}$. The following lemma, although very simple, is a key lemma to understand the approach considered in this paper.
\begin{lemma}\label{lem: dual of inverse}
Let $\mathrsfs{A}$ be an inverse automaton, then
$$
\partial\left(\mathrsfs{A}\sqcup \mathrsfs{A}^{-1}\right)=(\partial\mathrsfs{A})^{-}
$$
\end{lemma}
\begin{proof}
By i) of Proposition \ref{prop: dual prop} $\partial\mathrsfs{A}$ is reversible. Moreover, in $\mathrsfs{A}$ there is an edge $p\mapright{a|b}q$ if and only if $p^{-1}\mapright{b|a}q^{-1}$ in $\mathrsfs{A}^{-1}$ if and only if $a\mapright{p|q}b$ and $b\vlongfr{p^{-1}}{q^{-1}}a$ are edges in $(\partial\mathrsfs{A})^{-}$.
\end{proof}
By enriching $\partial\mathrsfs{A}$ with the a structure of inverse transducer allows us to have a powerful tool to  geometrically encode algebraic and topological properties of the group $\mathcal{G}(\mathrsfs{A})$. For this reason we name $(\partial\mathrsfs{A})^{-}$ as the \emph{enriched dual} of $\mathrsfs{A}$.
We have the following characterization of the relations defining the group $\mathcal{G}(\mathrsfs{A})$.
\begin{theorem}\label{theo: charact relations}
Let $\mathrsfs{A}=(Q,A,\cdot,\circ)$ be an invertible transducer, with $\mathcal{G}(\mathrsfs{A})\simeq F_{Q}/N$. Consider the transducer $(\partial\mathrsfs{A})^{-}=(A,\wt{Q},\circ,\cdot)$, and let
\begin{equation}\label{eq: containing relation}
\mathcal{N}\subseteq \bigcap_{a\in A}L\left((\partial\mathrsfs{A})^{-},a\right)
\end{equation}
be the maximal subset invariant for the action $A\overset{\cdot}{\curvearrowright} \wt{Q}^{*}$.Then $N=\oo{\mathcal{N}}$.
\end{theorem}
\begin{proof}
Let $\mathcal{N}'=\sigma^{-1}(N)$, then for any $u\in\mathcal{N}'$ we have
\begin{eqnarray}
\label{eq: relation} u\circ (a_{1}\ldots a_{n})=a_{1}\ldots a_{n}\\
\label{eq: relation2} \left[u\cdot(a_{1}\ldots a_{n})\right]\circ (b_{1}\ldots b_{m})=b_{1}\ldots b_{m}
\end{eqnarray}
for any $a_{1}\ldots a_{n}, b_{1},\ldots, b_{m}\in A^{*}$ holds in $\mathrsfs{A}\sqcup \mathrsfs{A}^{-1}=(\wt{Q},A,\cdot, \circ)$. In particular, equations (\ref{eq: relation}) (\ref{eq: relation2}) hold for single elements $a_{1},b_{1}$, hence by Lemma \ref{lem: dual of inverse}, $u\in \cap_{a\in A}L\left((\partial\mathrsfs{A})^{-},a\right)$, and by (\ref{eq: relation2}) $a_{1}\cdot u\in \mathcal{N}'$ which shows that $\mathcal{N}'$ is stable for  $A\overset{\cdot}{\curvearrowright} \wt{Q}^{*}$. Whence by maximality $\mathcal{N}'\subseteq \mathcal{N}$.
Conversely, assume that (\ref{eq: containing relation}) holds, and let us prove by induction on $k$ that $u\circ (a_{1}\ldots a_{k})=a_{1}\ldots a_{k}$. Note that $\mathcal{N}$ is stable for $A\overset{\cdot}{\curvearrowright} \wt{Q}^{*}$ if and only if it is stable for $A^{*}\overset{\cdot}{\curvearrowright} \wt{Q}^{*}$. Since $u\in\cap_{a\in A}L\left((\partial\mathrsfs{A})^{-},a\right)$, then for any $a\in A$ the induction base $u\circ a=a$ clearly holds. Assume now that $u\circ (a_{1}\ldots, a_{k-1})=a_{1}\ldots, a_{k-1}$ hence
$$
u\circ (a_{1}\ldots a_{k})=a_{1}\ldots a_{k-1}\left [u\cdot (a_{1}\ldots a_{k-1})\right] \circ a_{k}=a_{1}\ldots a_{k}
$$
where in the last passage we have used the stability of $ \mathcal{N}$ for the action $A^{*}\overset{\cdot}{\curvearrowright} \wt{Q}^{*}$, i.e. $(a_{1}\ldots a_{k-1})\cdot u\in \mathcal{N}$.
\end{proof}
By $\mathcal{C}^{+}(G,A)$ we denote the \emph{the Cayley graph} of the group $G$ with generating set $A$, while we denote by $\mathcal{C}(G,A)$ the \emph{Cayley automaton} of the group $G$ with generating set $A\cup A^{-1}$ and base point the identity of $G$. Note that if $G=F_{A}/H$, then $L(\mathcal{C}(G,A))=\sigma^{-1}(H)$. The following theorem gives a way to represent the Schreier automata from the powers of the enriched dual as well as a way to build (in the limit) the Cayley automaton of the group $\mathcal{G}(\mathrsfs{A})$.
\begin{theorem}\label{thm:schreier}
Let $G=\mathcal{G}(\mathrsfs{A})=F_{Q}/N$. For any $k\ge 1$ put $\mathcal{D}_{k}=((\partial\mathrsfs{A})^{-})^{k}_{\mathcal{I}}$. The following facts hold.
\begin{enumerate}
\item[i)] If $v=a_{1}\ldots a_{k}\in A^{k}$, and $H=\St_{G}(v)$. Then
$$
\left(\mathcal{D}_{k}, v\right)\simeq (\Sch(H,A),H)
$$
\item[ii)] If $\ul{v}=a_{1}a_{2}\ldots\in A^{\omega}$, and $H=\St_{G}(v)$, then $\varprojlim \{\mathcal{D}^{k}_{\mathcal{I}}\}_{k\ge 1}=\mathcal{D}_{\mathcal{I}}^{\infty}$ is an inverse graph such that
$$
(\mathcal{D}_{\mathcal{I}}^{\infty}, \ul{v})\simeq  (\Sch(H,A),H)
$$
\item[iii)] Let $(\mathcal{N}_{k},v_{k})=\prod_{v\in A^{k}} (\mathcal{D}_{k},v)$. Then $\mathcal{C}(G/\St_{G}(k),Q)\simeq (\mathcal{N}_{k},v_{k})$.
Furthermore, the inclusions $L(\mathcal{N}_{k},v_{k})\subseteq L(\mathcal{N}_{k-1},v_{k-1})$ induces maps $\psi_{k,k-1}:(\mathcal{N}_{k},v_{k})\rf (\mathcal{N}_{k-1},v_{k-1})$ giving rise to the inverse system $(\{(\mathcal{N}_{k},v_{k})\}_{k\ge 1}, \psi_{i,j})$ such that if $\mathcal{N}^{\infty}=\varprojlim \{(\mathcal{N}_{k},v_{k})\}_{k\ge 1}$, then
$$
\mathcal{C}(G,Q)\simeq \mathcal{N}^{\infty}
$$
\end{enumerate}
\end{theorem}
\begin{proof}
i) By Proposition \ref{prop: immersion lang hom} it is enough to prove that $L(\mathcal{D}_{k}, v)=L(\Sch(H,A),H)$. Let $u\in L(\mathcal{D}_{k}, v)$, then $u\circ v=v$, and so $\oo{u}\in H$, whence $u\in L(\Sch(H,A),H)$. Conversely, if $u\in L(\Sch(H,A),H)$, then $\oo{u}\in H$, whence $u\circ v=v$, i.e. $u\in L(\mathcal{D}_{k},v)$.\\
ii) As before, we have to prove that $L(\mathcal{D}_{\mathcal{I}}^{\infty}, \ul{v})=  L(\Sch(H,A),H)$. Given $\ul{v}\in A^{\omega}$, we denote by $H_k=\St_G(\ul{v}[k])$. We have
\begin{eqnarray*}
u\in  L(\mathcal{D}_{\mathcal{I}}^{\infty}, \ul{v}) &\Leftrightarrow & \forall \ k\geq 1, u\in  L(\mathcal{D}_{\mathcal{I}}^{\infty}, \ul{v}[k])\\
 & \Leftrightarrow & u\in L(\Sch(H_k,A),H_k) \\
 & \Leftrightarrow & u \in L(\Sch(H,A),H).\\
\end{eqnarray*}
The last equivalence follows from the fact that if $u \not\in L(\Sch(H,A),H)$, then $uH\neq H$. This implies there is $h\in H$ such that $uh\not\in H$. In particular, since $H=\cap H_k$ there exists $k$ with $uh\not\in H_k$. Hence $uH_k\neq H_k$. On the other hand. if $u\in L(\Sch(H,A),H)$ it follows $u\in L(\Sch(H_k,A),H_k)$ for any $k$.\\
iii) The first stamens follows by showing $L(\mathcal{N}_{k},v_{k})=\sigma^{-1}(\St_{G}(k))$ and Proposition \ref{prop: immersion lang hom}. From the definition of product of automata we have $u\in L(\prod_{v\in A^{k}} (\mathcal{D}_{k},v))$ if and only if $v\circ u=v$ for every $v\in A^k$ which is equivalent to $u\in \sigma^{-1}(\St_G(k))$.\\
Let us show that the map $\psi_{i,j}$ is well defined for $i\geq j$. This follows from the fact that if $u\in \sigma^{-1}(\St_G(i))$ then $u\circ v=v$ for every $v=v_1\cdots v_j\cdots v_i$. In particular $u\circ (v_1\cdots v_j) =v_1\cdots v_j$ and this implies $u\in \sigma^{-1}(\St_G(j))$.\\
The isomorphism in the statement follows from Proposition \ref{prop: immersion lang hom} and $L(\mathcal{N}^{\infty})=\sigma^{-1}(N)$. Indeed, we get
 \begin{eqnarray*}
u\in  \sigma^{-1}(N) & \Leftrightarrow & \forall \ k\geq 1, \ \forall v\in A^k  \  u\in L(\mathcal{D}_{k}, v)\\
 &\Leftrightarrow & \forall \ k\geq 1 \  u\in L(\mathcal{N}_{k},v_{k}) \\
  &\Leftrightarrow &   u\in \bigcap_{k\ge 1}L(\mathcal{N}_{k},v_{k})=L(\mathcal{N}^{\infty})
\end{eqnarray*}
where the last equality follows from the usual properties of projective limit.
\end{proof}
Note that by the previous theorem $\Core(\mathcal{N}^{\infty})$ is not complete if and only if $\mathcal{G}(\mathrsfs{A})$ is free.
We now state the analogous of Theorem \ref{theo: charact relations} for the semigroup $\mathcal{S}(\mathrsfs{A})$ defined by the transducer $\mathrsfs{A}=(Q,A,\cdot,\circ)$. First we need the following definition. We say that a pair of words $u,v\in Q^{*}$ are \emph{colliding} in $\partial\mathrsfs{A}=(A,Q,\circ,\cdot)$ whenever $a\circ u=a\circ v$ for all $a\in A$. A set $\mathcal{R}\subseteq Q^{*}\times Q^{*}$ is said to be invariant for the action $A\overset{\cdot}{\curvearrowright} Q^{*}$ whenever $a\cdot (u,v)=(a\cdot u, a\cdot v)\in \mathcal{R}$ for all $(u,v)\in \mathcal{R}$, $a\in A$.
\begin{theorem}\label{theo: relations semigroup}
Let $\mathrsfs{A}=(Q,A,\cdot,\circ)$ be an inverse automaton, with $\mathcal{S}(\mathrsfs{A})\simeq Q^{*}/\mathcal{R}$. Consider the automaton $\partial\mathrsfs{A}=(A,Q,\circ,\cdot)$. Then, $\mathcal{R}$ is the maximal subset of $Q^{*}\times Q^{*}$ of colliding pairs invariant for the action $A\overset{\cdot}{\curvearrowright} Q^{*}$.
\end{theorem}
\begin{proof}
In the same spirit of the proof of Theorem \ref{theo: inverse-groups} the statement follows from the observation that $u=v$ in $\mathcal{S}(\mathrsfs{A})$ if and only if
\begin{eqnarray}
\nonumber a\circ u=a\circ v\\
\nonumber a\circ\left[u\cdot b\right]= a\circ \left[v\cdot b\right]
\end{eqnarray}
holds for any $a,b\in A$ in $\partial\mathrsfs{A}$.
\end{proof}

\section{Connected components of powers of transducers}\label{sec: finiteness}
In this section we study finiteness conditions. In particular we show that the finiteness of the semigroup $\mathcal{S}(\mathrsfs{A})$ is equivalent to the finiteness of $\mathcal{S}(\partial\mathrsfs{A})$. We also show, by using the notion of supersymmetric transducers, a relationship between the geometry of the the powers of the dual automata $\partial\mathrsfs{A}$ and the finiteness of the group generated by $\mathrsfs{A}$. More precisely, transducers whose powers of the dual keep to be symmetric and, in some sense, homogeneous, generate only a finite number of actions of the original group.
\\
We now study how the connected components of two inverse transducers behaves when we compose them.  Let $\mathrsfs{A}=(Q,A,\cdot,\circ)$ be a transducer, note that for any $q\in Q^{k}$ and $u\in A^{m}$ the automata $\left (\mathrsfs{A}^{k}_{\mathcal{I}},q \right)$ and $\left ((\partial\mathrsfs{A})^{m}_{\mathcal{I}},u \right)$ represent the graph of the orbits of $q,u$ with respect to the the actions $Q^{*}\overset{\cdot}{\curvearrowleft} A^{*}$, $A^{*}\overset{\circ}{\curvearrowleft}Q^{*}$, respectively. We have the following proposition.
\begin{prop}\label{prop: finiteness}
Let $\mathrsfs{A}=(Q,A,\cdot,\circ)$ be a transducer. Then the following are equivalent:
\begin{enumerate}
 \item[i)] $\mathcal{S}(\mathrsfs{A})$ is finite;
 \item[ii)] there is an integer $n$ such that $\|(\partial \mathrsfs{A})^{k}\|\le n$ for all $k\ge 1$.
 \item[iii)] $\mathcal{S}(\partial\mathrsfs{A})$ is finite;
 \item[iv)] there is an integer $m$ such that $\|\mathrsfs{A}^{k}\|\le m$ for all $k\ge 1$.
\end{enumerate}
\end{prop}
\begin{proof}
i) $\Rightarrow$ ii). Suppose that $\mathcal{S}(\mathrsfs{A})$ is finite and consider the orbits of $\mathcal{S}(\mathrsfs{A})$ on the sets $A^k$. These orbits correspond to the connected components of the graphs of the orbit $\left ((\partial\mathrsfs{A})^{k}_{\mathcal{I}},u_{0} \right)$, $\left ((\partial\mathrsfs{A})^{k}_{\mathcal{I}},u_1 \right),\ldots, \left ((\partial\mathrsfs{A})^{k}_{\mathcal{I}},u_{n(k)} \right)$ for some $n(k)$. Since $\mathcal{S}(\mathrsfs{A})$ is finite there exists $n$ such that $\|\left((\partial\mathrsfs{A})^{k}_{\mathcal{I}},u_{i} \right)\|\leq n$ for any $k\ge 1$ and $i=1,\ldots, n(k)$. In particular, since $(\partial \mathrsfs{A})^{k}$ is represented by the disjoint union of the automata $\left ((\partial\mathrsfs{A})^{k},u_i \right)$, $\|(\partial \mathrsfs{A})^{k}\|\leq n$ for every $k$.\\
ii) $\Rightarrow$ iii). Suppose that $\|(\partial \mathrsfs{A})^{k}\|\leq n$ for every $k\geq 1$. This implies that up to isomorphism there is only a finite number $N$ of automata $\left ((\partial\mathrsfs{A})^{k}_{\mathcal{I}},u_i \right)$. Notice that such graphs describe the complete action of the dual semigroup $\mathcal{S}(\partial\mathrsfs{A})$ on $Q^{\ast}$. More precisely, the number of possible different actions, i.e. the number of distinct elements in $\mathcal{S}(\mathrsfs{A})$ is uniformly bounded by $N(n|Q|)^{|Q|^{2}}$.  This implies that $\mathcal{S}(\mathrsfs{A})$ is finite.\\
iii) $\Rightarrow$ iv). It follows from i) $\Rightarrow$ ii) by passing to the dual transducer.\\
iv) $\Rightarrow$ i). It follows from ii) $\Rightarrow$ iii) by passing to the dual transducers.
\end{proof}
From the previous proposition we get the following result.
\begin{cor}
Let $\mathrsfs{A}=(Q,A,\cdot,\circ)$ be an invertible transducer and put $G=\mathcal{G}(\mathrsfs{A})$. The following are equivalent:
\begin{enumerate}
 \item[i)] $G$ is finite;
 \item[ii)] $\mathcal{S}(\partial\mathrsfs{A})$ is finite;
  \item[iii)] the cardinality of all the Schreier graphs in the boundary are upper bounded, i.e. there is an integer $n$ such that $\|\Sch((\St_{G}(\ul{u}),A),\St_{G}(\ul{u}))\|\le n$ for all $\ul{u}\in A^{\omega}$.
\end{enumerate}
\end{cor}
\begin{proof}
Observe that, in the case of groups, it follows from Proposition \ref{prop: inverse of products} and Theorem \ref{thm:schreier} that the graphs of the orbits $\left((\partial\mathrsfs{A})^{k}_{\mathcal{I}}, u\right)$ of $(\partial\mathrsfs{A})^{k}$ seen as inverse automata are given by the Schreier graphs $\Sch(H,A)$ and there is a bijection between the connected components of $\Sch(H,A)$ and $(\partial \mathrsfs{A})^{k}$. Proposition \ref{prop: finiteness} implies that $G$ is finite if and only if $\|(\partial \mathrsfs{A})^{k}\|\leq n$ for every $k$. In particular, we have that $\|\Sch((\St_{G}(\ul{u}),A),\St_{G}(\ul{u}))\|\le n$ for all $\ul{u}\in A^{\omega}$. On the other hand, if there is a sequence of components $((\partial \mathrsfs{A})^{k},v_1\cdots v_k)$ such that $|((\partial \mathrsfs{A})^{k},v_1\cdots v_k)|>k$, then the orbit of $v=v_1v_2\cdots \in A^{\omega}$ is unbounded. This implies $\|\Sch((\St_{G}(\ul{u}),A),\St_{G}(\ul{u}))\|=\infty$.
\end{proof}

\begin{prop}\label{prop: norm of product}
Let $\mathrsfs{A},\mathrsfs{B}$ be two inverse transducers, and let $p,q$ be two vertices belonging to $\mathrsfs{A},\mathrsfs{B}$, respectively. The following conditions are equivalent:
\begin{enumerate}
\item [i)]
$
\|\left(\mathrsfs{A}\mathrsfs{B}, (p,q)\right)\|= \|(\mathrsfs{A},p)\|
$
\item[ii)]
$
 \mathrsfs{A}_{p}\left(L(\mathrsfs{A}_{\mathcal{I}},p)\right)\subseteq L(\mathrsfs{B}_{\mathcal{I}},q)
$
\item[iii)]
$
L\left(\mathrsfs{A}\mathrsfs{B}, (p,q)\right)=L(\mathrsfs{A}_{\mathcal{I}},p)
$
\end{enumerate}
\end{prop}
\begin{proof}
i) $\Leftrightarrow$ ii) It is straightforward to check that $\|\left(\mathrsfs{A}\mathrsfs{B}, (p,q)\right)\|\ge \|(\mathrsfs{A},p)\|$ holds. Note that the condition $\|\left(\mathrsfs{A}\mathrsfs{B}, (p,q)\right)\|> \|(\mathrsfs{A},p)\|$ is clearly equivalent in having two distinct states $(p',t)$, $(p',t')$ of $\mathrsfs{A}\mathrsfs{B}$ in the same connected component of $(p,q)$. This occurs if and only if there are two words $u,v\in \wt{A}^{*}$ such that $p\longfr{u}{u_{1}}p'$, $p\longfr{v}{v_{1}}p'$ are paths in $\mathrsfs{A}$, with $u_{1}=\mathrsfs{A}_{p}(u)$, $v_{1}=\mathrsfs{A}_{p}(v)$, and there are paths $q\vlongfr{u_{1}}{u_{2}}t$, $q\vlongfr{v_{1}}{v_{2}}t'$ in $\mathrsfs{B}$. Since the two transducers are inverse, these last conditions are equivalent to the fact that $p\vvvlongfr{uv^{-1}}{u_{1}v_{1}^{-1}}p\mbox{ is a path in } \mathrsfs{A}$
while $q\vvvlongfr{u_{1}v_{1}^{-1}}{u_{2}v_{2}^{-1}}q$ is not a path in $\mathrsfs{B}$, i.e. $\mathrsfs{A}_{p}\left(L(\mathrsfs{A}_{\mathcal{I}},p)\right)\nsubseteq L(\mathrsfs{B}_{\mathcal{I}},q)$.

\noindent Equivalence ii) $\Leftrightarrow$ iii) is a consequence of the following equality:
$$
L\left(\mathrsfs{A}\mathrsfs{B}, (p,q)\right)=L(\mathrsfs{A}_{\mathcal{I}},p)\cap \mathrsfs{A}_{p}^{-1}\left(L(\mathrsfs{B}_{\mathcal{I}},q)\right)
$$
\end{proof}
\noindent If $\mathrsfs{A}$ is invertible we have the stronger condition
$$
\|\left(\mathrsfs{A}\mathrsfs{B}, (p,q)\right)\|\ge\max\{\|(\mathrsfs{A},p)\|, \|(\mathrsfs{B},p)\|\}
$$
In this case we have the following proposition whose proof is similar to the previous one and it is left to the reader.
\begin{prop}\label{prop: norm of product invertible}
Let $\mathrsfs{A},\mathrsfs{B}$ be two inverse transducers, with $\mathrsfs{A}$ invertible, and let $p,q$ be two vertices belonging to $\mathrsfs{A},\mathrsfs{B}$, respectively. The following conditions are equivalent:
\begin{enumerate}
\item [i)]
$$
\|\left(\mathrsfs{A}\mathrsfs{B}, (p,q)\right)\|= \max \left\{\|(\mathrsfs{A},p)\|, \|(\mathrsfs{B},q)\|\right\}
$$
\item[ii)] \begin{eqnarray}
\nonumber \mathrsfs{A}_{p}\left(L(\mathrsfs{A}_{\mathcal{I}},p)\right)\subseteq L(\mathrsfs{B}_{\mathcal{I}},q)\mbox{ if }\|(\mathrsfs{B},q)\|\le \|(\mathrsfs{A},p)\|\\
\nonumber L(\mathrsfs{B}_{\mathcal{I}},q)\subseteq \mathrsfs{A}_{p}\left(L(\mathrsfs{A}_{\mathcal{I}},p)\right)\mbox{ if }\|(\mathrsfs{A},p)\|\le \|(\mathrsfs{B},q)\|
\end{eqnarray}
\item[iii)] \begin{eqnarray}
\nonumber L\left(\mathrsfs{A}\mathrsfs{B}, (p,q)\right)&=&L(\mathrsfs{A}_{\mathcal{I}},p)\mbox{ if }\|(\mathrsfs{B},q)\|\le \|(\mathrsfs{A},p)\|\\
\nonumber L\left(\mathrsfs{A}\mathrsfs{B}, (p,q)\right)&=&\mathrsfs{A}_{p}^{-1}\left(L(\mathrsfs{B}_{\mathcal{I}},q)\right)\mbox{ if }\|(\mathrsfs{A},p)\|\le \|(\mathrsfs{B},q)\|
\end{eqnarray}
\end{enumerate}
\end{prop}
\noindent Note that we can also use the shorter notation $\mathrsfs{A}_{p}\left(L(\mathrsfs{A}_{\mathcal{I}},p)\right)=L(\mathrsfs{A}_{\mathcal{O}},p)$. We say that a connected component $(\mathrsfs{B},q)$ of an inverse transducer $\mathrsfs{B}$ has the \emph{swapping inclusion property} with respect to the pair $(p_{1},p_{2})$ if the condition of Proposition \ref{prop: norm of product} holds for $\mathrsfs{A}=\mathrsfs{B}$, i.e. $L(\mathrsfs{B}_{\mathcal{O}}, p_{1})\subseteq L(\mathrsfs{B}, p_{2})$. Furthermore, if $\mathrsfs{B}$ is invertible, we say that $(\mathrsfs{B},q)$  has the \emph{swapping invariant property} with respect to the pair $(p_{1},p_{2})$ if the conditions of Proposition \ref{prop: norm of product invertible} hold: $L(\mathrsfs{B}_{\mathcal{O}}, p_{1})= L(\mathrsfs{B}, p_{2})$, if in addition this condition holds for any pair $(p_{1},p_{2})$ of $(\mathrsfs{B},q)$, then $(\mathrsfs{B},q)$ is named \emph{supersymmetric}. Note that by Proposition \ref{prop: immersion lang hom}, the supersymmetry condition implies that for any pair $p_{1}, p_{2}$ of vertices $(\mathrsfs{B}_{\mathcal{I}}, p_{1})\simeq (\mathrsfs{B}_{\mathcal{I}}, p_{2})$. There is another symmetry given by the swapping invariance condition as the following proposition shows.
\begin{prop}\label{prop: swapping is isomorphism}
Let $\mathrsfs{B}=(Q,A,\cdot, \circ)$ be an inverse transducer and let $(\mathrsfs{B},q)$ be a connected component which is swapping invariant with respect to $(p,p)$, then $(\mathrsfs{B}_{\mathcal{O}}, p)\simeq (\mathrsfs{B}_{\mathcal{I}}, p)$.
\end{prop}
\begin{proof}
In general $(\mathrsfs{B}_{\mathcal{O}},p)$ is not deterministic, so let $\oo{\mathrsfs{B}_{\mathcal{O}}}$ be the inverse graph obtained from $\mathrsfs{B}_{\mathcal{O}}$ applying all the possible foldings. We claim that
$$
L(\mathrsfs{B}_{\mathcal{O}}, p)=L(\oo{\mathrsfs{B}_{\mathcal{O}}}, p)
$$
Since $L(\mathrsfs{B}_{\mathcal{O}}, p)\subseteq L(\oo{\mathrsfs{B}_{\mathcal{O}}}, p)$ holds, we just need to prove the other inclusion. Thus, let $u\in L(\oo{\mathrsfs{B}_{\mathcal{O}}}, p)$. A path $p\path{u}p$ in $\oo{\mathrsfs{B}_{\mathcal{O}}}$ lifts to a path $p\path{u'}p$ in $\mathrsfs{B}_{\mathcal{O}}$ with
$$
u'=v_{1}u_{1}v_{2}u_{2}\ldots v_{n-1}u_{n-1}v_{n}
$$
for some words $v_{i}, u_{i}\in \wt{A}^{*}$ with $u=u_{1}u_{2}\ldots u_{n-1}$ and $\oo{v_{1}}=\oo{v_{2}}=\ldots=\oo{v_{n}}=1$. Since $\mathrsfs{B}_{\mathcal{I}}$ is an inverse graph, being $\mathrsfs{B}$ an inverse transducer, and $u'\in L(\mathrsfs{B}_{\mathcal{O}}, p)= L(\mathrsfs{B}_{\mathcal{I}}, p)$ we also have that $u\in L(\mathrsfs{B}_{\mathcal{I}}, p)=L(\mathrsfs{B}_{\mathcal{O}}, p)$, whence the claim $L(\oo{\mathrsfs{B}_{\mathcal{O}}}, p)\subseteq L(\mathrsfs{B}_{\mathcal{O}}, p)$. Since $L(\mathrsfs{B}_{\mathcal{O}}, p)=L(\oo{\mathrsfs{B}_{\mathcal{O}}}, p)$, and by the minimality property of Proposition \ref{prop: immersion lang hom}, we have that $(\mathrsfs{B}_{\mathcal{I}}, p)\simeq(\oo{\mathrsfs{B}_{\mathcal{O}}}, p)$, whence no folding is performed. Thus, $(\oo{\mathrsfs{B}_{\mathcal{O}}}, p)=(\mathrsfs{B}_{\mathcal{O}}, p)$, and so the statement $(\mathrsfs{B}_{\mathcal{O}}, p)\simeq (\mathrsfs{B}_{\mathcal{I}}, p)$.
\end{proof}
\noindent By Proposition \ref{prop: norm of product invertible} it is obvious that $(\mathrsfs{B},q)$ is swapping invariance with respect to $(p_{1},p_{2})$ if and only if $\|\left(\mathrsfs{B}\mathrsfs{B}, (p_{1},p_{2})\right)\|= \|(\mathrsfs{B},q)\|$, and $(\mathrsfs{B},q)$ is supersymmetric if and only if $\|\left(\mathrsfs{B}\mathrsfs{B}, (p_{1},p_{2})\right)\|= \|(\mathrsfs{B},q)\|$ for all the vertices $(p_{1},p_{2})$ of $(\mathrsfs{B},q)$. The notion of supersymmetric component is related to finite groups defined by $RI$-transducers as we will show. For this reason we define the \emph{symmetric powers} $\mathrsfs{S}_{j}$ $j\ge 0$, of an $RI$-transducer $\mathrsfs{A}$ inductively by $\mathrsfs{S}_{0}=(\partial\mathrsfs{A})^{-}$, and $\mathrsfs{S}_{i}=\mathrsfs{S}_{i-1}^{2}$ for all $i>0$.
We have the following lemma.
\begin{lemma}\label{lem: charact bireav powers}
For any $j\ge 0$, $\mathrsfs{S}_{j}$ is bireversible if and only if $\mathrsfs{S}_{j+1}$ is bireversible. In particular, the product of a non-bireversible component $(\mathrsfs{S}_{j},q)$ having two distinct edges $q\mapright{a|b}p'$,$q'\mapright{c|b}p'$ with any component $(\mathrsfs{S}_{j},s)$ gives rise to a non-bireversible component $\left(\mathrsfs{S}_{j+1},(q,s)\right)$.
\end{lemma}
\begin{proof}
By Proposition \ref{prop: product reversible} we just need to prove the ``if'' part. Thus, suppose that $\mathrsfs{S}_{j+1}=\mathrsfs{S}_{j}^{2}$ is bireversible. We can assume that $\mathrsfs{S}_{j}$ is not bireversible with two distinct edges $q\mapright{a|b}p'$, $q'\mapright{c|b}p'$ in $(\mathrsfs{S}_{j},q)$, and consider any edge $s\mapright{b|d}s'$ in $(\mathrsfs{S}_{j},s)$. Hence $(p,s)\mapright{a|d}(p',s')$, $(t,s)\mapright{c|d}(p',s')$ are two distinct edges in $\left(\mathrsfs{S}_{j+1},(q,s)\right)$, i.e. $\mathrsfs{S}_{j+1}$ is not bireversible, a contradiction.
\end{proof}
\noindent The following proposition shows that for finite groups defined by $RI$-transducers the maximal connected components of the symmetric powers eventually become supersymmetric.
\begin{prop}\label{prop: maximal are supersymmetric}
Let $\mathrsfs{A}$ be an $RI$-transducer such that $G=\mathcal{G}(\mathrsfs{A})$ is finite. With the above notation, there is an integer $n$ such that for all $i\ge n$ the maximal components of $\mathrsfs{S}_{i}$ are supersymmetric.
\end{prop}
\begin{proof}
Since $\|\mathrsfs{S}_{j+1}\|\ge \|\mathrsfs{S}_{j}\|$, by Proposition \ref{prop: finiteness} there is an integer $n$ such $\|\mathrsfs{S}_{i}\|= \|\mathrsfs{S}_{n}\|$ for all $i\ge n$. Thus for any maximal connected components $(\mathrsfs{S}_{i},q)$ we get $\| (\mathrsfs{S}_{i},q)\|=\|\mathrsfs{S}_{n}\|$ for all $i\ge n$. In particular for any pair $(p_{1},p_{2})$ of $(\mathrsfs{S}_{i},q)$ we get:
$$
\| \left(\mathrsfs{S}_{i+1},(p_{1},p_{2})\right)\|=\| \left(\mathrsfs{S}_{i},p_{1}\right)\|
$$
Thus, by Proposition \ref{prop: norm of product invertible} $(\mathrsfs{S}_{i},q)$ has the swapping invariant property with respect to the pair $(p_{1},p_{2})$, i.e. $(\mathrsfs{S}_{i},q)$ is supersymmetric.
\end{proof}

Using the previous results we obtain an alternative proof of the following corollary.
\begin{cor}\cite[Corollary 22]{Pic12}\label{cor:piccantino}
Let $\mathrsfs{A}$ be an $RI$-transducer. If $\mathrsfs{A}$ is not bireversible, then $\mathcal{G}(\mathrsfs{A})$ is infinite.
\end{cor}
\begin{proof}
Suppose contrary to our claim that $G=\mathcal{G}(\mathrsfs{A})$ is finite. Let $n$ be the integer of Proposition \ref{prop: maximal are supersymmetric} such that for all $i\ge n$ the maximal connected components of $\mathrsfs{S}_{i}$ are supersymmetric. By Lemma \ref{lem: charact bireav powers} there is an integer $j\ge n$ and a maximal connected component $(\mathrsfs{S}_{j},q)$ which is supersymmetric and not bireversible, hence there are two edges $p\mapright{a|b}p'$,$t\mapright{c|b}p'$ in $(\mathrsfs{S}_{j},q)$. By Proposition \ref{prop: swapping is isomorphism} $\left((\mathrsfs{S}_{j})_{\mathcal{I}},p'\right)\simeq \left((\mathrsfs{S}_{j})_{\mathcal{O}},p'\right)$, i.e. $\left((\mathrsfs{S}_{j})_{\mathcal{O}},p'\right)$ is an inverse automaton. However, in $(\mathrsfs{S}_{j})_{\mathcal{O}}$ there are the two edges $p\mapright{b}p'$,$t\mapright{b}p'$, a contradiction.
\end{proof}

\section{Some applications to Cayley types of machines}\label{sec: cayley}
In this section we show some applications of the results obtained by the approach that we have developed in Section \ref{sec: dual}. We focus our attention to duals of transducers whose input automata are Cayley graphs, this class contains the Cayley machines, already studied from the spectral point of view in \cite{KaSiSte}. In the same flavor of \cite{DaRo2013} we consider ``colorings'' of the Cayley graph $\mathcal{C}^{+}(G,A)$ of the finite group $G$ with set of generators $A$ into transducers. Since $\mathcal{C}^{+}(G,A)$ is a reversible automaton all these colorings give rise to reversible transducers, hence by Proposition \ref{prop: dual prop} the duals of these machines are all invertible, and so they all define a group. In this section we consider the following Cayley machines types of transducers.
\begin{itemize}
\item The (``usual'') Cayley machine $\mathrsfs{C}(G)=(G,G,\circ,\cdot)$ with transitions of the form $g\longfr{x}{gx}gx$ for all $g,x\in G$.
\item The palindrome Cayley machine $\mathrsfs{C}_{p}(G)=(G,G,\circ,\cdot)$ with transitions of the form $g\longfr{x}{x^{-1}}gx$.
\item The identity Cayley machine $\mathrsfs{C}_{I}(G)=(G,G,\circ,\cdot)$ with transitions of the form $g\longfr{x}{x}gx$ for all $g,x\in G$.
\end{itemize}
Note that $\mathrsfs{C}_{p}(G),\mathrsfs{C}_{I}(G)$ are bireversible.
Let $G=F_{G}/H$, when dealing with the enriched dual to avoid confusion between inverses of $G$ and formal inverses of $G^{-1}$, we shall write words in $\wt{G}^{*}$ using parenthesis. For instance, $(u^{-1})$ represents the element $u^{-1}$ of $G$, while $(u)^{-1}$ is the element of $G^{-1}$ which is the formal inverse of $u$. We denote by $e$ the identity of the group $G$. For a word $(u_{1})^{e_{1}}\ldots (u_{k})^{e_{k}}=u\in \wt{G}^{*}$ with $e_{1},\ldots e_{k}\in\{1,-1\}$ we put $\hat{u}=u_{1}^{e_{1}}\ldots u_{k}^{e_{k}}\in G^{*}$. Note that the following holds
\begin{equation}\label{eq: positive cycles}
u\in\sigma^{-1}(H)\quad \Longleftrightarrow \quad u\in L(\mathrsfs{C}(G)^{-},g)\quad \Longleftrightarrow \quad u\in L(\mathcal{C}(G,G))\quad \Longleftrightarrow\quad \hat{u}=e
\end{equation}
\\
In \cite[Theorem 4.1]{SiSte} the authors describe the groups defined by Cayley machines, in particular it is proven that the semigroups defined by these machines are free. The following theorem proves the analogous fact for the dual of these machines.
\begin{theorem}\label{theo: freeness dual cayley}
For any non-trivial finite group $G$, then the semigroup $\mathcal{S}(\partial\mathrsfs{C}(G))$ is free.
\end{theorem}
\begin{proof}
Let $\mathcal{S}(\partial\mathrsfs{C}(G))=G^{*}/\mathcal{R}$ we claim $\mathcal{R}$ is the identity relation. Indeed, assume by contradiction, that $\mathcal{R}$ is not the identity relation. By \cite[Lemma 2.7]{SiSte} we can assume that there is an element $(u,v)\in \mathcal{R}$ with $|u|=|v|$ and $u\neq v$. Furthermore, we can assume that $|u|$ is minimal among all the pairs $(s,t)\in \mathcal{R}$ with $|s|=|t|$. By Theorem \ref{theo: relations semigroup} we are seeking for a maximal subset $\mathcal{R}\subseteq G^{*}\times G^{*}$ of colliding pairs which is invariant for the action $G\overset{\cdot}{\curvearrowright} G^{*}$ in $\mathrsfs{C}(G)$. Note that since the input-automaton of $\mathrsfs{C}(G)$ is the Cayley automaton of $G$ with generating set $G$, then for any colliding pair $(s,w)$ with $s=s_{1}\ldots s_{m}$, $w=w_{1}\ldots w_{k}$ we get
\begin{equation}\label{eq: condition colliding cayley}
s_{1}\ldots s_{m}=w_{1}\ldots w_{k}
\end{equation}
Since the set $\mathcal{R}$ is invariant for the action $G\overset{\cdot}{\curvearrowright} G^{*}$ and $\mathrsfs{C}(G)$ is invertible, we have that $\mathcal{R}$ is also invariant for the action $(G)^{-1}\overset{\cdot}{\curvearrowright} G^{*}$ given by the inverse automaton $\mathrsfs{C}(G)^{-1}=(G^{-1},G,\circ, \cdot)$ with transitions $(g)^{-1}\vlongfr{y}{g^{-1}y} (y)^{-1}$. Thus, in particular for the chosen pair $(u,v)$ with $u=u_{1}\ldots u_{n}$, $v=v_{1}\ldots v_{n}$ and any $g\in G$ the pair $(u',v')$ with
$$
v'=(g)^{-1}\cdot \left(v_{1}\ldots v_{n}\right)=(g^{-1}v_{1})(v_{1}^{-1}v_{2})(v_{2}^{-1}v_{3})\ldots (v_{n-1}^{-1}v_{n})
$$
$$
u'=(g)^{-1}\cdot \left(u_{1}\ldots u_{n}\right)=(g^{-1}u_{1})(u_{1}^{-1}u_{2})(u_{2}^{-1}u_{3})\ldots (u_{n-1}^{-1}u_{n})
$$
belongs to $\mathcal{R}$. Hence, by (\ref{eq: condition colliding cayley}) we get $g^{-1}v_{n}=g^{-1}u_{n}$, i.e. $v_{n}=u_{n}$. Since $\partial\mathrsfs{C}(G)$ is invertible, then $\mathcal{S}(\partial\mathrsfs{C}(G))$ is a cancellative monoid, thus since $u_{1}\ldots u_{n}=v_{1}\ldots v_{n}$ in $\mathcal{S}(\partial\mathrsfs{C}(G))$, we get $u'=u_{1}\ldots u_{n-1}=v_{1}\ldots v_{n-1}=v'$ in $\mathcal{S}(\partial\mathrsfs{C}(G))$, i.e. $(u',v')\in \mathcal{R}$ with $|u'|<|u|$ and $u'\neq v'$, a contradiction, whence $\mathcal{R}$ is trivial.
\end{proof}
We now make some consideration regarding the group $\mathcal{G}(\partial\mathrsfs{C}(G))=F_{G}/N$. In this case we consider the automaton $\mathrsfs{C}(G)^{-}=(A,\wt{G},\circ,\cdot)$, and by Theorem \ref{theo: charact relations} we have to look for the maximal subset $\mathcal{N}\subseteq \cap_{g\in G}L(\mathrsfs{C}(G)^{-},g)$ which is stable for the action $G\overset{\cdot}{\curvearrowright} \wt{G}^{*}$. The set $\mathcal{N}$ represents the set of relations for the group $\mathcal{G}(\partial\mathrsfs{C}(G))$. By the previous Theorem \ref{theo: freeness dual cayley} we immediately obtain the following corollary.
\begin{cor}\label{cor: not positive occurrences}
With the above notations $\mathcal{N}\cap G^{+}=\emptyset$.
\end{cor}
The natural question is if there are non-trivial relations involving necessarily elements from $G^{-1}$.  A word $(u_{1})^{e_{1}}\ldots (u_{2k})^{e_{2k}}\in \wt{G}^{*}$ with $e_{1}=1$, $e_{i+1}=-e_{i}$ for $i=1,\ldots, 2k-1$ is called \emph{alternating}. We have the following lemma.
\begin{lemma}\label{lem: reduced cayley}
With the above notation, $\oo{g\cdot (u_{1})(u_{2})^{-1}}=1$ for any $g\in G$. In particular, $\oo{g\cdot u}=1$ for any alternating word.
\end{lemma}
\begin{proof}
In $\mathrsfs{C}(G)^{-}$ a transition with its inverse is given by $(gx^{-1})\vlongfr{x}{g}g$, $g\vvvlongfr{(x)^{-1}}{(g)^{-1}}(gx^{-1})$. Therefore, a simple computation yields to
$$
\oo{g\cdot(u_{1})(u_{2})^{-1}}=\oo{(gu_{1})(gu_{1})^{-1}}=1
$$
The last statement is a consequence of this last fact.
\end{proof}
Given a word $u=(u_{1})(u_{2})\ldots (u_{2k-1})(u_{2k})\in G^{*}$, we define the alternating map by $\alpha(u)=(u_{1})(u_{2}^{-1})^{-1}\ldots (u_{2k-1})(u_{2k}^{-1})^{-1}$. Note that $\alpha(u)$ is reduced if $u$ does not contain any factor $(v)(v^{-1})$. The following proposition shows that $\mathcal{G}(\partial\mathrsfs{C}(G))$ is in general not free.
\begin{prop}
With the above notation, for any $u\in \sigma^{-1}(H)$ with $|u|$ even, we have $\alpha(u)\in \mathcal{N}$. \end{prop}
\begin{proof}
By (\ref{eq: positive cycles}) $\sigma^{-1}(H)= L(\mathrsfs{C}(G)^{-},g)$ holds for all $g\in G$ and the alternating map preserves this inclusions, i.e. $\alpha(u)\in L(\mathrsfs{C}(G)^{-},g)$ for any $u\in \sigma^{-1}(H)$. Furthermore, if $|u|$ is even, $\alpha(u)$ is alternating. Hence, by Lemma \ref{lem: reduced cayley}, $\oo{h\cdot \alpha(u)}=1$ for all $h\in G^{*}$, i.e. $\alpha(u)\in \mathcal{N}$.
\end{proof}
For a word $u\in \wt{G}^{*}$ we may write $[u]_{c}=\{xy:yx=u\}$ for the set of all the words that can be obtained by a cyclic shift of $u$. In the next proposition we provide a recursive way to build the relations $\mathcal{N}$ that defines the group $\mathcal{G}(\partial\mathrsfs{C}(G))$. We shall use the two following subsets
$$
\mathcal{N}_{2k}=\mathcal{N}\cap \wt{G}^{2k}, \quad \mathcal{V}_{2k}=\{v\in \wt{G}^{2k}:\forall g\in G, g\cdot v\in \mathcal{N}_{2k}\}
$$
Note that given $\mathcal{N}_{2k}$ it is possible to (algorithmically) compute $\mathcal{V}_{2k}$ by
$$
\mathcal{V}_{2k}=\bigcap_{g\in G}\{u\in \wt{G}^{2k}: g\cdot u\in \mathcal{N}_{2k}\}
$$
Therefore, the previous equality and the following proposition provides a recursive way to compute the set $\mathcal{N}$.
\begin{prop}
With the above notation:
$$
\mathcal{N}_{2(k+1)}=\bigcup_{\substack{x,y\in G\\v\in \mathcal{V}_{2k}\\xy^{-1}\hat{v}=e}}\left[(x)(y)^{-1}v\right]_{c}
$$
\end{prop}
\begin{proof}
Let $T_{2(k+1)}$ the right side of the equality. We show the inclusion $\mathcal{N}_{2(k+1)}\supseteq T_{2(k+1)}$. It is enough to prove $(x)(y)^{-1}v\in \mathcal{N}_{2(k+1)}$ since by Theorem \ref{theo: charact relations} relations are invariant under cyclic shifts. Since $xy^{-1}\hat{v}=e$, by (\ref{eq: positive cycles}) it is easy to see that
$$
(x)(y)^{-1}v\in L(\mathrsfs{C}(G)^{-},g), \quad \forall g\in G
$$
Moreover, by Lemma \ref{lem: reduced cayley} for any $g\in G$ we have:
$$
g\cdot (x)(y)^{-1}v=(gx)(gx)^{-1}\left(gxy^{-1}\right)\cdot v\in \mathcal{N}_{2(k+1)}
$$
since by definition $(gx)(gx)^{-1}\in\mathcal{N}_{2}$, $\left(gxy^{-1}\right)\cdot v\in \mathcal{N}_{2k}$ and $\mathcal{N}$ is closed for the concatenation of words. Hence, $(x)(y)^{-1}v\in \mathcal{N}_{2(k+1)}$.
\\
Conversely, consider a relation $u\in \mathcal{N}_{2(k+1)}$ by Corollary \ref{cor: not positive occurrences} $u$ contains a positive ($G$) and a negative ($G^{-1}$) occurrence of $G$, i.e. $u=w(r)w'(s)^{-1}w''$. Thus, after applying a suitable cyclic shift we can assume $u'=(x)(y)^{-1}v\in  \mathcal{N}_{2(k+1)}$ for some $v\in\wt{G}^{2k}$. Hence, if we prove $u'\in T_{2(k+1)}$, then since $T_{2(k+1)}$ is closed by cyclic shifts, we also have $u\in T_{2(k+1)}$. By (\ref{eq: positive cycles}) we clearly have $xy^{-1}\hat{v}=e$.
Furthermore, by Lemma \ref{lem: reduced cayley} $(gx)(gx)^{-1}\left(gxy^{-1}\right)\cdot v=g\cdot (x)(y)^{-1}v\in \mathcal{N}_{2(k+1)}$. Hence, taking $h=gxy^{-1}$ we have $h\cdot v\in  \mathcal{N}_{2k}$ for all $h\in G$, whence $v\in \mathcal{V}_{2k}$ and so $u'=(x)(y)^{-1}v\in T_{2(k+1)}$.
\end{proof}
We now pass to study the last two cases starting from the the palindrome Cayley machine. The name comes from the notion of palindrome, these are words which are equal to their mirror images. Given a subset $L\subseteq \wt{G}^{*}$, its mirror $L^{R}$ is the set of words $u=u_{1}\ldots u_{n}\in \wt{G}^{*}$ whose mirrors $u^{R}=u_{n}\ldots u_{1}\in L$. Note that the mirror of a reduced word is still reduced. Therefore, if $N\le F_{G}$, the set $N^{R}=\sigma\left ((\sigma^{-1}(N)^{R})\right)$ is still a subgroup, and in case $N$ is normal, $N^{R}$ is also normal. Note that $\oo{u}\in N\cap N^{R}$ if and only if $\oo{u}^{R}\in N$. Consider the group $G=F_{G}/H$, the \emph{associated palindromic group} is the group
$$
G^{(p)}=F_{G}/(H\cap H^{R})
$$
Note that this group gives an index of the degree of ``palindromicity'' of $H$. Indeed, in case $H^{R}=H$, we get $G=G^{(p)}$, while in the other extreme case $H\cap H^{R}=\{e\}$ we get $G^{(p)}=F_{G}$ is free.
\begin{prop}\label{prop: palindromic}
With the above notation $K=\mathcal{G}(\partial\mathrsfs{C}_{p}(G))\simeq G^{(p)}$. Moreover, for any element $u\in G$ $\St_{K}(u)=H/\left(H\cap H^{R}\right)$, and for $u\in G^{\ge 2}$, $\St_{K}(u)=\{1\}$.
\end{prop}
\begin{proof}
Let $K=F_{G}/N$, by Theorem \ref{theo: charact relations} $N=\oo{\mathcal{N}}$, where
$$
\mathcal{N}\subseteq \bigcap_{g\in G}L\left(\mathrsfs{C}_{p}(G)^{-},g\right)
$$
is the maximal subset invariant for the action $G\overset{\cdot}{\curvearrowright} \wt{G}^{*}$. By (\ref{eq: positive cycles}) we have $L\left(\mathrsfs{C}_{p}(G)^{-},g\right)=\sigma^{-1}(H)$. We claim that $N=H\cap H^{R}$. Let $u=(u_{1})^{e_{1}}\ldots (u_{n})^{e_{n}}$ for some $e_{1},\ldots, e_{n}\in\{1,-1\}$ such that $u=\oo{u}\in H\cap H^{R}$. Hence for any $g\in G$ we get
$$
g\cdot u=g\cdot \left((u_{1})^{e_{1}}\ldots (u_{n})^{e_{n}}\right)=(u_{1}^{-1})^{e_{1}}\ldots (u_{n}^{-1})^{e_{n}}\in H\cap H^{R}
$$
since $(u_{1}^{-1})^{e_{1}}\ldots (u_{n}^{-1})^{e_{n}}\in H$ and $(g\cdot u)^{R}=(u_{n}^{-1})^{e_{n}}\ldots (u_{1}^{-1})^{e_{1}}\in H$ by the fact that $(u^{R})^{-1}=(u_{1})^{e_{1}-1}\ldots (u_{n})^{e_{n}-1}\in H$ and $(u)^{-1}=(u_{n})^{e_{n}-1}\ldots (u_{1})^{e_{1}-1}\in H$, respectively. Moreover, assume $\sigma^{-1}(H\cap H^{R})\subsetneq \mathcal{N}$ and let $u\in \mathcal{N}\setminus \sigma^{-1}(H\cap H^{R})$ with $u=\oo{u}$. Thus $u\in H$, and for any $g\in G$ $g\cdot u=(u_{1}^{-1})^{e_{1}}\ldots (u_{n}^{-1})^{e_{n}}\in H$, hence $(u_{n}^{-1})^{e_{n}-1}\ldots (u_{1}^{-1})^{e_{1}-1}\in H$, which implies $u^{R}=(u_{n})^{e_{n}}\ldots (u_{1})^{e_{1}}\in H$, i.e. $u\in H\cap H^{R}$.
\\
Let $\pi: F_{G}\rightarrow G^{p}=F_{G}/(H\cap H^{R})$ be the canonical map, and put $\mathcal{D}_{m}=\left(\mathrsfs{C}_{p}(G)^{-}\right)^{m}_{\mathcal{I}}$. Note that for $m=1$ and for all $g\in G$, $\sigma\left(L(\mathcal{D}_{1}, g)\right)=H$, while it is not difficult to check that for $m\ge 2$ and for all $g\in G^{m}$ we have $\sigma\left(L(\mathcal{D}_{m}, g)\right)\subseteq H\cap H^{R}$. Hence, by Theorem \ref{thm:schreier}, for any $g\in G$ we get
$$
\St_{K}(g)=\pi\left(\sigma\left(L(\Sch(\St_{K}(g),G),\St_{K}(g))\right)\right)=\pi\left(\sigma(L(\mathcal{D}_{1}, g))\right) =H/(H\cap H^{R})
$$
while for any $g\in G^{m}$ with $m\ge 2$ we get
$$
\St_{K}(g)=\pi\left(\sigma\left(L(\Sch(\St_{K}(g),G),\St_{K}(g))\right)\right)=\pi\left(\sigma(L(\mathcal{D}_{m}, g))\right) =\{1\}
$$
\end{proof}
We have the following immediate corollary of the previous proposition.
\begin{cor}
For a palindromic group $G=G^{(p)}$, $G\simeq\mathcal{G}(\partial\mathrsfs{C}_{(p)}(G))$. Moreover, for any element $u\in G^{*}\sqcup G^{\omega}$, $\St_{G}(u)$ is trivial.
\end{cor}
The following propositions shows that for any finite group $G$ there is a transducer defining $G$ and having all trivial stabilizers.
\begin{prop}
For any finite group $G$, $K=\mathcal{G}(\partial\mathrsfs{C}_{I}(G))\simeq G$. Moreover, for any element $g\in G^{*}$, $\St_{K}(g)$ is trivial.
\end{prop}
\begin{proof}
Let $G=F_{G}/H$. The proof follows the same line of the proof of Proposition \ref{prop: palindromic}. In this case the action $G\overset{\cdot}{\curvearrowright} \wt{G}^{*}$ is the identity, and $\sigma(L\left(\mathrsfs{C}_{I}(G)^{-},g\right))=H$ for any $g\in G$, from which it follows that $\mathcal{G}(\partial\mathrsfs{C}_{I}(G))\simeq F_{G}/H$. Analogously, putting $\mathcal{D}_{m}=\left(\mathrsfs{C}_{I}(G)^{-}\right)^{m}_{\mathcal{I}}$, then for any $g=g'g''$, for some $g'\in G, g''\in G^{*}$, we get $\sigma(L\left(\mathcal{D}_{m},g\right))\subseteq \sigma\left(L\left(\mathrsfs{C}_{I}(G)^{-},g'\right)\right)\subseteq H$, whence $\St_{K}(g)=\{1\}$.
\end{proof}
The techniques developed for these two last cases can be easily adapted to study the more general class of Cayley type of machine $\mathrsfs{C}_{\varphi}(G)=(G,A,\circ,\cdot)$ depending on a map $\varphi:A\rightarrow A$, and with transitions of the form $g\vlongfr{x}{\varphi(x)}gx$ for all $g\in G$, $x\in A$.

\section{Dynamics on the boundary}\label{sec: dynamics}
In this section we are interested in the dynamic of a group defined by a transducer acting on the boundary of the associated rooted tree. More precisely we deal with the problem of determine the algebraic properties of a group acting on the boundary having all trivial stabilizers. In the second part we deal with the existence of finite Schreier graphs on the boundary, and we finally give an algorithm to solve the problem of finding finite Schreier graphs on the boundary whose dimension is upper bounded by a certain integer passed in the input. Our approach strongly uses the machinery developed in the previous sections.

\subsection{Stabilizers on the boundary} In \cite{GriSa13} Grigorchuk and Savchuk noticed that in all known essentially free actions of not virtually abelian groups generated by finite automata there is at least one singular
point. Recall that if a group $G$ acts on a topological space $X$, a singular point is a point such that $\St_G(x)\neq \St^0_G(x)$, where $\St^0_G(x)$ denotes the subgroup of elements acting trivially on some neighborhood of $x$.
\\
In Section \ref{sec: cayley} we have seen some examples of groups having trivial stabilizers in the boundary. However, all these examples are finite.
An element $\xi\in A^{\omega}$ is called \emph{periodic} if $\xi=w^{\omega}$ for some $w\in A^{*}$. We may denote the set of periodic points as $A^{\omega}_{p}$. Periodic points play an important role in this context, as the following proposition shows.
\begin{prop}\label{prop: density prop of periodic points}
Let $\mathrsfs{A}=(Q,A,\cdot,\circ)$ be an invertible transducer, and let $G=\mathcal{G}(\mathrsfs{A})$. Then, for any $\ul{u}\in A^{\omega}$, $\St_{G}(\ul{u})=\{1\}$ if and only if $\St_{G}(\ul{v})$ for all $\ul{v}\in A^{\omega}_{p}$.
\end{prop}
\begin{proof}
Let $\pi: \wt{A}^{*}\rightarrow G$ be the canonical map. The ``only if'' part is trivial, so let us prove the ``if'' part. Assume that there is an element $\ul{u}\in A^{\omega}$ such that $\St_{G}(\ul{u})$ is non-trivial, and let $g\in \St_{G}(\ul{u})$ with $g\neq 1$ represented by an element $q_{1}\ldots q_{k}\in \wt{Q}^{*}$. Thus, for any $n>0$ we have
\begin{eqnarray}
\nonumber (q_{1}\ldots q_{k})\circ \ul{u}[n]&=&\ul{u}[n]\\
\nonumber (q_{1}\ldots q_{k})\cdot \ul{u}[n]&=&v_{n}
\end{eqnarray}
for some $v_{n}\in\wt{Q}^{k}$. Suppose that for some $m$, $\pi(v_{m})=1$, hence it is straightforward to check
$$
(q_{1}\ldots q_{k})\circ \ul{u}[m]^{\omega}=\ul{u}[m]^{\omega}
$$
whence $g\in\St_{G}((\ul{u}[m])^{\omega})$, and the statement holds. Therefore, we can assume $\pi(v_{m})\neq 1$ for all $m>0$. By the finiteness of $\wt{Q}^{k}$, and since $\{v_{i}\}_{i>0}$ is infinite, there are two indices $j_{1}, j_{2}$ such that $v_{j_{1}}=v_{j_{2}}$. It is straightforward to check that
$$
v_{j_{1}}\circ \ul{u}[j_{1}+1, j_{2}]^{\omega}=\ul{u}[j_{1}+1, j_{2}]^{\omega}
$$
Since $\pi(v_{j_{1}})\neq 1$, we get $\pi(v_{j_{1}})\in \St_{G}(\ul{u}[j_{1}+1, j_{2}]^{\omega})\neq \{1\}$.
%
\end{proof}

We recall that an element $a$ of a semigroup is of finite order if the \emph{monogenic} subsemigroup $\la a\ra$ generated by $a$ is finite. In this case there are two integers $k,p$, the \emph{index} and the \emph{period}, respectively, such that $a^{k+p}=a^{k}$. The set of elements of finite order is denoted by $\To(S)$, and $S$ is \emph{torsion-free} if this set is empty. In case we consider the semigroup $\mathcal{S}(\mathrsfs{A})$ generated by the transducer $\mathrsfs{A}=(Q,A,\cdot,\circ)$, if $\pi:Q^{*}\rightarrow \mathcal{S}(\mathrsfs{A})$
denotes the canonical map, with a slight abuse of notation we say $q\in Q^{*}$ is an element of $ \To(\mathcal{S}(\mathrsfs{A})) $ whenever $\pi(q)\in \To(\mathcal{S}(\mathrsfs{A})) $. We have the following lemma.
\begin{lemma}\label{lem: candidates of stabilizers}
Let $\mathrsfs{A}=(Q,A,\cdot,\circ)$ be a (non necessarily invertible) reversible transducer, and let $q\in Q^{k}$ such that $q\in\To(\mathcal{S}(\mathrsfs{A}))$. Put $\mathrsfs{A}^{k}=(Q^{k},A,\cdot,\circ)$, and let $n$ be the cardinality of the transition monoid (which is a group) of $(Q^{k},A,\cdot)$. If $m,p$ are the index and the period of $\pi(q)$, then there is an integer $0<\ell\le n^{m+p-1}$ such that $q^{\omega}\cdot u^{\ell}=q^{\omega}$.
\end{lemma}
\begin{proof}
By Proposition \ref{prop: product reversible} $\mathrsfs{A}^{k}$ is reversible, hence the transition monoid $(Q^{k},A,\cdot)$ is a group and we clearly have $q\cdot u^{n}=q$. Throughout the proof we will use the following fact which can be easily proved: if $q\cdot v=q$, then for any integer $i$ we have $(q\circ v)^{i}=q\circ v^{i}$. By definition of index and period we get that for all $j\ge 1$ and $0\le r<p$ we have
\begin{equation}\label{eq: period index}
q^{m+jp+r}\circ v=q^{m+r}\circ v,\:\:\forall v\in A^{*}
\end{equation}
We claim that there is an integer $\ell\le n^{m+p-1}$ such that for all $0\le i\le m+p-1$ we have:
\begin{equation}\label{eq: power to loops}
q\cdot(q^{i}\circ (u^{\ell})) =q
\end{equation}
Since (\ref{eq: power to loops}) holds for $i=0$, with $\ell=n$, let $s$ be the maximum integer with $0\le s\le m+p-1$ satisfying (\ref{eq: power to loops}) with $\ell_{s}\le n^{s}$. Assume, contrary to the claim, that $s <m+p-1$. Therefore, $q\cdot(q^{s+1}\circ (u^{\ell_{s}})) \neq q$, but since $\mathrsfs{A}^{k}$ is reversible, we get $q\cdot (q^{s+1}\circ (u^{\ell_{s}}))^{n}=q$. By the previous remark, from $q\cdot (q^{s}\circ (u^{\ell_{s}}))=q$ we get $(q^{s+1}\circ (u^{\ell_{s}}))^{n}=q\circ (q^{s}\circ (u^{\ell_{s}}))^{n}$. Hence, $q\cdot(q^{s+1}\circ (u^{\ell_{s}}))^{n} =q$, and using an induction, by (\ref{eq: power to loops}) and the same remark we can move the $n$-th power to the exponent of $u$, whence $(q^{i}\circ (u^{\ell_{s}}))^{n}=q\circ(q^{i-1}\circ (u^{\ell_{s}}))^{n}$ holds for all $0\le i\le s$. In particular for $i=s$ we get
$$
(q^{s+1}\circ (u^{\ell_{s}}))^{n}=q\circ (q^{s}\circ (u^{\ell_{s}}))^{n}=q^{s+1}\circ (u^{\ell_{s}n})
$$
Taking $\ell=\ell_{s} n\le n^{s+1}$, we have that  (\ref{eq: power to loops}) holds for $s+1$, a contradiction. We claim that $q^{m+jp}\cdot u^{\ell}=q^{m+jp}$ holds for all $j\ge 1$. By (\ref{eq: power to loops}), the induction base $q^{m+p}\cdot u^{\ell}=q^{m+p}$ clearly holds. Hence, assume the claim true for $j>1$, then we get
$$
q^{m+(j-1)p}=q^{m+(j-1)p}\cdot u^{\ell}=\left(q^{(j-1)p}\cdot (q^{m}\circ u^{\ell})\right)q^{m}
$$
whence, from this equation, (\ref{eq: period index}) and $q^{m+p}\cdot u^{\ell}=q^{m+p}$ we obtain:
\begin{eqnarray}
\nonumber q^{m+jp}\cdot u^{\ell}&=&\left(q^{(j-1)p}\cdot\left (q^{m+p}\circ u^{\ell}\right)\right)(q^{m+p}\cdot u^{\ell})=\left(q^{(j-1)p}\cdot\left (q^{m}\circ u^{\ell}\right)\right)q^{m+p}\\
\nonumber &=&\left(q^{(j-1)p}\cdot\left (q^{m}\circ u^{\ell}\right)q^{m}\right)q^{p}=\left(q^{m+(j-1)p}\right)q^{p}=q^{m+jp}
\end{eqnarray}
from which it follows the statement of the lemma $q^{\omega}\cdot u^{\ell}=q^{\omega}$.
\end{proof}
The following celebrated result is due to Zelmanov \cite{Zelma90,Zelma91}.
\begin{theorem}\label{zelmanov}
Any finitely generated residually finite group with bounded torsion is finite.
\end{theorem}
The following theorem relates torsion elements of the semigroup defined by the enriched dual of an invertible transducer with the stabilizers of periodic elements in the boundary.
\begin{theorem}\label{prop: torsion free}
Let $\mathrsfs{A}=(Q,A,\cdot,\circ)$ be an invertible transducer such that $G=\mathcal{G}(\mathrsfs{A})$ is not finite. Then for any $u\in A^{*}$ with $u\in \To(\mathcal{S}(\left(\partial\mathrsfs{A}\right)^{-}))$ we have $\St_{G}(u^{\omega})\neq \{1\}$.
\end{theorem}
\begin{proof}
Assume by contradiction that there is some $u\in \To(\mathcal{S}(\left(\partial\mathrsfs{A}\right)^{-}))$ with $u\in A^{k}$ such that  $\St_{G}(u^{\omega})= \{1\}$. Let $\pi:A^{*}\rightarrow \mathcal{S}\left((\partial\mathrsfs{A})^{-}\right)$ be the canonical map, and let $m,p$ be the index and the period of the element $\pi(u)$, respectively, and let $n$ be the order of the transition monoid (which is a group) of $\left((\partial\mathrsfs{A})^{-}\right)^{k}_{\mathcal{I}}$. We claim that $G$ is a torsion group and that the order of the elements is bounded by $n^{m+p-1}$. Thus, let $\psi: \wt{Q}^{*}\rightarrow G$ be the canonical map, take any $g\in G$ and $q\in \wt{Q}^{*}$ such that $\psi(q)=g$. By Lemma \ref{lem: candidates of stabilizers} applied to $(\partial\mathrsfs{A})^{-}$, we get that there is an integer $\ell\le n^{m+p-1}$ such that $u^{\omega}\circ q^{\ell}=u^{\omega}$. Hence,  $g^{\ell}\in\St_{G}(u^{\omega})=\{1\}$, i.e. the claim $g^{\ell}=1$ with $\ell\le n^{m+p-1}$. Therefore by Theorem \ref{zelmanov}, being $G$ residually finite, $G$ must be finite, a contradiction.
\end{proof}
In particular, looking for examples of groups defined by automata with trivial stabilizers in the boundary is restricted to transducers whose enriched dual generates a torsion-free semigroup. We record this fact in the following corollary.
\begin{cor}
Let $\mathrsfs{A}=(Q,A,\cdot,\circ)$ be an invertible transducer such that $G=\mathcal{G}(\mathrsfs{A})$ is not finite, and $\St_{G}(u^{\omega})=\{1\}$ for all $u\in A^{*}$, then $\mathcal{S}\left((\partial\mathrsfs{A})^{-}\right)$ is torsion-free.
\end{cor}

\subsection{Finite Schreier graphs} We now consider the study the existence of finite Schreier graphs in the boundary. We need first a lemma which is a  consequence of the results proved in Section \ref{sec: finiteness}, for this reason we stick to the notation introduced in that section.
\begin{lemma}\label{lem: swapping inclusion}
Let $\mathrsfs{B}=(Q,\wt{A},\cdot,\circ)$ be an inverse transducer which has the swapping inclusion property with respect to the pair $(p_{1},p_{2})$. Then, there is a sequence $q_{0},\ldots, q_{m}$ of distinct vertices of $(\mathrsfs{B},p_{1})$ such that:
$$
L(\mathrsfs{B}_{\mathcal{O}},q_{i})\subseteq L(\mathrsfs{B},q_{i+1})
$$
for $i=0,\ldots, m-1$, and $L(\mathrsfs{B}_{\mathcal{O}},q_{m})\subseteq L(\mathrsfs{B},q_{0})$.
\end{lemma}
\begin{proof}
Let $p_{1}\longfr{h_{1}}{h_{1}'}p_{2}$ be a path in $(\mathrsfs{B},p_{1})$ connecting $p_{1}$ with $p_{2}$ for some $h_{1}\in \wt{A}^{*}$. Put $h_{2}=h_{1}'$, and let $p_{3}$ be the vertex such that $p_{2}\longfr{h_{2}}{h_{2}'}p_{3}$ is a path in $(\mathrsfs{B},p_{1})$. We claim that $L(\mathrsfs{B}_{\mathcal{O}},p_{2})\subseteq L(\mathrsfs{B},p_{3})$. Thus, let $v'\in L(\mathrsfs{B}_{\mathcal{O}},p_{2})$, and let $v\in  L(\mathrsfs{B},p_{2})$ with $p_{2}\longfr{v}{v'}p_{2}$. Hence there is a path
$$
p_{1}\exlongfr{h_{1}vh_{1}^{-1}}{h_{2}v'h_{2}^{-1}}p_{1}
$$
Since $(\mathrsfs{B},p_{1})$ has the swapping inclusion property with respect to the pair $(p_{1},p_{2})$ we get $h_{2}v'h_{2}^{-1}\in L(\mathrsfs{B},p_{2})$, and so using the fact that $\mathrsfs{B}$ is inverse we get $v'\in L(\mathrsfs{B},p_{3})$, and this prove the claim $L(\mathrsfs{B}_{\mathcal{O}},p_{2})\subseteq L(\mathrsfs{B},p_{3})$. Continuing in this way we can find a sequence $p_{1},p_{2},\ldots, p_{k},\ldots$ of vertices of $(\mathrsfs{B},p_{1})$ such that $L(\mathrsfs{B}_{\mathcal{O}},p_{i})\subseteq L(\mathrsfs{B},p_{i+1})
$ for all $i\ge 0$. Since $(\mathrsfs{B},p_{1})$ is finite, let $j,m+1$ be the integers such that $p_{j}=p_{j+m+1}$, the statement is satisfied considering the sequence $q_{0}=p_{j}, q_{2}=p_{j+1}, \ldots, q_{m}=p_{j+m}$.
\end{proof}
Henceforth we use the notation $H_{\ul{v}}=\St_{G}(\ul{v})$ for any $\ul{u}\in A^{\omega}$. We have the following sufficient conditions to have finite Schreier graphs for a periodic point in the boundary.
\begin{prop}
Let $\mathrsfs{A}=(Q,A,\cdot,\circ)$ be an invertible transducer, and let $\mathrsfs{C}=(\partial\mathrsfs{A})^{-}=(A,\wt{Q},\circ,\cdot)$. If a connected component $(\mathrsfs{C}^{k},p)$ has the swapping inclusion property, then there is a periodic point $\underline{v}\in A^{\omega}_{p}$ such that $\Sch((H_{\ul{v}},A),H_{\ul{v}})$ is finite.
\end{prop}
\begin{proof}
Let $\mathrsfs{B}=\mathrsfs{C}^{k}$, and let $q_{0},\ldots, q_{m}$ be the vertices of $(\mathrsfs{B},p)$ as in Lemma \ref{lem: swapping inclusion}. Put $w=q_{0}\ldots q_{m}$, and consider $\ul{v}=w^{\omega}$, we claim that $\Sch((H_{\ul{v}},A),H_{\ul{v}})$ is finite. We prove by induction that for any $j\ge 0$ the following inclusion holds:
\begin{equation}\label{eq: induction inclusion}
\mathrsfs{B}_{q_{j}\ldots q_{0}}\left(L(\mathrsfs{B}^{j+1},(q_{0},\ldots,q_{j}))\right)\subseteq L(\mathrsfs{B},q_{(j+1)\: mod\:m})
\end{equation}
Note that by Proposition \ref{prop: norm of product}, the previous statement is equivalent to
\begin{equation}\label{eq: swapping equality}
L(\mathrsfs{B}^{j+2},(q_{0},\ldots,q_{j+1}))=L(\mathrsfs{B}^{j+1},(q_{0},\ldots,q_{j}))
\end{equation}
Since $\mathrsfs{B}_{q_{0}}\left(L(\mathrsfs{B},q_{0})\right)\subseteq L(\mathrsfs{B},q_{1})$, the base case is satisfied. Thus, suppose the claim true for $j$ and let us prove it for $j+1$. Thus, by (\ref{eq: induction inclusion}) and (\ref{eq: swapping equality}) we get the following inclusions:
\begin{eqnarray}
\nonumber  \mathrsfs{B}_{q_{j+1}}\mathrsfs{B}_{q_{j}\ldots q_{0}}\left(L(\mathrsfs{B}^{j+2},(q_{0},\ldots,q_{j+1}))\right)=\mathrsfs{B}_{q_{j+1}}\mathrsfs{B}_{q_{j}\ldots q_{0}}\left(L(\mathrsfs{B}^{j+1},(q_{0},\ldots,q_{j}))\right)\subseteq\\
\nonumber \subseteq\mathrsfs{B}_{q_{j+1}}\left( L(\mathrsfs{B},q_{(j+1)\: mod\:m})\right) = L(\mathrsfs{B}_{\mathcal{O}},q_{(j+1)\: mod\:m})\subseteq  L(\mathrsfs{B},q_{(j+2)\: mod\:m})
\end{eqnarray}
hence (\ref{eq: induction inclusion}) holds for $j+1$. By (\ref{eq: swapping equality}) we have $L(\mathrsfs{B}^{\infty},w^{\omega})=L(\mathrsfs{B},q_{0})$, whence by Theorem \ref{thm:schreier} $(\mathrsfs{B}^{\infty},w^{\omega})\simeq \Sch((H_{\ul{v}},A),H_{\ul{v}})$ is finite.
\end{proof}
To have a result which also holds in the other direction we need to look at points in the boundary that generalize the notion of periodic points. We recall that a point $w\in A^{\omega}$ is called \emph{almost periodic} (sometimes called also \emph{pre-periodic}) if there is $x,y\in A^{*}$ such that $w=xy^{\omega}$. With this notion we characterize the case when we have a finite Schreier graph in the boundary. First we need to introduce a growth function. Let $\mathrsfs{A}=(Q,A,\cdot,\circ)$ be an invertible transducer, we consider the following \emph{components growth function} defined by
$$
\chi_{\mathrsfs{A}}(n)=\min\{\|\left(((\partial\mathrsfs{A})^{-})^{n}, q\right)\|: q\in A^{n})\}
$$
Note that $\chi_{\mathrsfs{A}}(n)$ is monotonically increasing, and by simple computation shows that $\chi_{\mathrsfs{A}}(n)\le \|\partial\mathrsfs{A}\|^{n}$ holds.
\begin{theorem}\label{theo: finiteness schreier}
Let $\mathrsfs{A}=(Q,A,\cdot,\circ)$ be an invertible transducer, and let $m\ge 1$ be an integer. Then, the following are equivalent
\begin{enumerate}
  \item[i)] There exists $\ul{v}\in A^{\omega}$ such that $\| \Sch((H_{\ul{v}},A),H_{\ul{v}}) \|=\chi_{\mathrsfs{A}}(m)$, and we have that $\Sch((H_{\ul{v}},A),H_{\ul{v}})$ is the smallest Schreier graph among the ones in the boundary;
  \item[ii)] There exists $\ul{v}\in A^{\omega}$ such that $\| \Sch((H_{\ul{v}},A),H_{\ul{v}}) \|\leq \chi_{\mathrsfs{A}}(m)$;
  \item[iii)] There exists an almost periodic element $\ul{v}=xy^{\omega}$ such that $|x|+|y|\le m+(|A|\|\partial\mathrsfs{A}\|^{m})^{|A|^{2}}$ and $\| \Sch((H_{\ul{v}},A),H_{\ul{v}}) \|\leq \chi_{\mathrsfs{A}}(m)$;
  \item[iv)] $\chi_{\mathrsfs{A}}(m)$=$\chi_{\mathrsfs{A}}(m+i)$ for all $ i\le  (|A|\|\partial\mathrsfs{A}\|^{m})^{|A|^{2}}$;
  \item[v)] $\chi_{\mathrsfs{A}}(m)$=$\chi_{\mathrsfs{A}}(m+i)$ for all $i\ge 0$.
\end{enumerate}
\end{theorem}
\begin{proof}
$i)\Leftrightarrow ii)$ is a consequence of the definitions. \\
$ii)\Rightarrow v)$. Suppose that $\lim_{i\to\infty}\chi_{\mathrsfs{A}}(m+i)=\infty$, hence by Theorem \ref{thm:schreier} all the Schreier graph are infinite, i.e. $\| \Sch((H_{\ul{v}},A),H_{\ul{v}}) \|=\infty$ for all $\ul{u}\in A^{\omega}$.\\
$v)\Rightarrow iv)$. Trivial.\\
$iv)\Rightarrow iii)$ Put $\mathrsfs{B}=(\partial\mathrsfs{A})^{-}$. Condition iv) implies that there is a connected component $(\mathrsfs{B}^{m}, (q_{1},\ldots,q_{m}))$ of cardinality $\chi_{\mathrsfs{A}}(m)$ and a finite sequence $q_{m+i}$ for $0\le i\le  (|A|\|\partial\mathrsfs{A}\|^{m})^{|A|^{2}}$ of vertices of $\mathrsfs{B}$ such that for all $0\le i\le  (|A|\|\partial\mathrsfs{A}\|^{m})^{|A|^{2}}$
$$
\|(\mathrsfs{B}^{m+i}, q_{1}\ldots q_{m+i}))\|=\chi_{\mathrsfs{A}}(m)
$$
Hence, by Proposition \ref{prop: norm of product} and Proposition \ref{prop: immersion lang hom} this sequence of transducers have all isomorphic input automata, i.e. for all $0\le i\le  (|A|\|\partial\mathrsfs{A}\|^{m})^{|A|^{2}}-1$ we have
\begin{equation}\label{eq: all isom input auto}
\left((\mathrsfs{B}^{m+i})_{\mathcal{I}}, q_{1}\ldots q_{m+i}\right)\simeq \left((\mathrsfs{B}^{m+i+1})_{\mathcal{I}}, q_{1}\ldots q_{m+i+1}\right)
\end{equation}
Now, if we see $\left(\mathrsfs{B}^{m+i}, q_{1}\ldots q_{m+i}\right)$ with $0\le i\le  (|A|\|\partial\mathrsfs{A}\|^{m})^{|A|^{2}}$, as inverse transducers, or equivalently as $\wt{A}\times\wt{A}$-automata, by (\ref{eq: all isom input auto}) and a straightforward computation it is easy to check that there are at most $(|A|\|\partial\mathrsfs{A}\|^{m})^{|A|^{2}}$ possible inverse transducers with $\|(\mathrsfs{B}^{m}, q_{1}\ldots q_{m}))\|$ states. Therefore, there are two integers $k,p$ with $k+p\le  (|A|\|\partial\mathrsfs{A}\|^{m})^{|A|^{2}}$ such that the following isomorphism (as $\wt{A}\times\wt{A}$-automata)
\begin{equation}\label{eq: repetition as transducers}
\left(\mathrsfs{B}^{m+k}, q_{1} \ldots q_{m+k}\right)\simeq \left(\mathrsfs{B}^{m+k+p}, q_{1} \ldots q_{m+k+p}\right)
\end{equation}
holds.
Consider the words $x=q_{1}\ldots q_{m+k}$, $y=q_{m+k+1}\ldots q_{m+k+p}$, and the almost periodic point $\ul{v}=xy^{\omega}$. Hence, we clearly have
$$
|x|+|y|\le m+(|A|\|\partial\mathrsfs{A}\|^{m})^{|A|^{2}}
$$
and we claim that $\|\Sch((H_{\ul{v}},A),H_{\ul{v}})\|\le \chi_{\mathrsfs{A}}(m)$ holds. Since (\ref{eq: repetition as transducers}) holds, then it is not difficult to check that
$$
\left(\mathrsfs{B}^{m+k+p}, xy\right)\simeq \left(\mathrsfs{B}^{m+k+tp}, xy^{t}\right)
$$
holds for any $t\ge 1$. Therefore, by (\ref{eq: all isom input auto}) we get $\left(\mathrsfs{B}^{\infty}_{\mathcal{I}}, xy^{\omega}\right)\simeq \left((\mathrsfs{B}^{m+k})_{\mathcal{I}}, x\right)$, whence by Theorem \ref{thm:schreier}
$$
 \left((\mathrsfs{B}^{m+k})_{\mathcal{I}}, x\right)\simeq \left(\mathrsfs{B}^{\infty}_{\mathcal{I}}, xy^{\omega}\right)\simeq \Sch((H_{\ul{v}},A),H_{\ul{v}})
$$
is a finite Schreier graph with $\|\Sch((H_{\ul{v}},A),H_{\ul{v}})\|=\|\left((\mathrsfs{B}^{m+k})_{\mathcal{I}}, x\right)\|=\chi_{\mathrsfs{A}}(m)$, and this concludes the proof of the implication.\\
$iii)\Rightarrow ii)$. Trivial.
\end{proof}
Put $c=\|\partial\mathrsfs{A}\|$ and let us define the following function:
$$
\zeta_{\mathrsfs{A}}(n)=\max\left\{y:n\ge |A|^{|A|^{2}}\left(\frac{c^{|A|^{2}(y+1)}-c^{|A|^{2}}}{c^{|A|^{2}}-1}\right)-\frac{y(y+1)}{2} \right\}
$$
It is easy to check that $\zeta_{\mathrsfs{A}}(n)$ is also monotonically increasing, and a rough lower bound for this function is given by
$$
\zeta_{\mathrsfs{A}}(n)\ge \frac{1}{|A|^{2}}\log_{c}\left(\frac{(c^{|A|^{2}}-1)n}{|A|^{|A|^{2}}}\right)-1
$$
From the previous theorem we obtain the following gap result.
\begin{cor}\label{cor: lower bound on chi}
Given an invertible transducer $\mathrsfs{A}$, then either $\chi_{\mathrsfs{A}}(n)$ stabilizes for a certain integer $m$, i.e. $\chi_{\mathrsfs{A}}(m+i)=\chi_{\mathrsfs{A}}(i)$ for all $i\ge 0$, or $\chi_{\mathrsfs{A}}(n)$ is not upper bounded and $ \chi_{\mathrsfs{A}}(n)\ge \zeta_{\mathrsfs{A}}(n)$. In particular, if there is a Schreier graph $\| \Sch((H_{v},A),H_{v}) \|< \zeta_{\mathrsfs{A}}(m)$ for some $v\in A^{m}$, then  there is an almost periodic point $\ul{v}\in A^{\omega}$ such that $\| \Sch((H_{\ul{v}},A),H_{\ul{v}}) \|\leq \chi_{\mathrsfs{A}}(m)$.
\end{cor}
\begin{proof}
By Theorem \ref{theo: finiteness schreier}, in case $\chi_{\mathrsfs{A}}(n)$ is not upper bounded, then for any $m$ we have
$$
\chi_{\mathrsfs{A}}\left(m+(|A|\|\partial\mathrsfs{A}\|^{m})^{|A|^{2}}-1\right)>\chi_{\mathrsfs{A}}(m)
$$
Thus, for a fixed integer $n$, $\chi_{\mathrsfs{A}}(n)$ is greater or equal to the greatest integer $y$ such that the following inequality
$$
n\ge \sum_{j=1}^{y}\left(|A|^{|A|^{2}}c^{|A|^{2}j}-1\right)=|A|^{|A|^{2}}\left(\frac{c^{|A|^{2}(y+1)}-c^{|A|^{2}}}{c^{|A|^{2}}-1}\right)-\frac{y(y+1)}{2}
$$
holds, whence $\chi_{\mathrsfs{A}}(n)\ge \zeta_{\mathrsfs{A}}(n)$. The last statement is a consequence of the previous results and Theorem \ref{thm:schreier}.
\end{proof}
Another consequence of the previous theorem is the following decidability result.
\begin{cor}
Given an invertible transducer $\mathrsfs{A}$, the following algorithmic question is decidable:
\begin{itemize}
\item ``Input'': An integer $\ell$.
\item ``Output'': Is there a Schreier graph $\Sch((H_{\ul{v}},A),H_{\ul{v}})$ for some element $\ul{v}\in A^{\omega}$ in the boundary with $\|\Sch((H_{\ul{v}},A),H_{\ul{v}}) \|\leq \ell$?
\end{itemize}
\end{cor}
\begin{proof}
Let $t$ be the smallest integer such that $\zeta_{\mathrsfs{A}}(t)>\ell$. Note that this integer is clearly computable. Furthermore, by Theorem \ref{theo: finiteness schreier} checking if there is a Schreier graph $\Sch((H_{\ul{v}},A),H_{\ul{v}})$ for some element $\ul{v}\in A^{\omega}$ in the boundary with $\|\Sch((H_{\ul{v}},A),H_{\ul{v}}) \|\leq \ell$ is equivalent to check if there is an integer $1\le j\le t$ for which condition $iv)$ holds and such that $\chi_{\mathrsfs{A}}(j)\le \ell$. This is clearly a decidable task. Moreover, if such integer does not exist, then by Corollary \ref{cor: lower bound on chi} we have $\chi_{\mathrsfs{A}}(j)\ge \zeta_{\mathrsfs{A}}(j)>\ell$ for all $j>t$, hence by $i)$ of Theorem \ref{theo: finiteness schreier} the smallest Schreier graph in the boundary is strictly greater than $\ell$.
\end{proof}

\section*{Acknowledgments}

The first author was supported by Austrian Science Fund project FWF P24028-N18.\\
The second author acknowledges support from the European Regional Development Fund through the programme COMPETE and by the Portuguese Government through the FCT under the project PEst-C/MAT/UI0144/2011 and the support of the FCT project SFRH/BPD/65428/2009.

\bibliographystyle{plain}
\bibliography{AutGroupsBiblio}

\begin{thebibliography}{10}

\bibitem{Pic12}
A.~Akhavi, I.~Klimann, S.~Lombardy, J.~Mairesse, and M.~Picantin.
\newblock On the finiteness problem for automaton (semi)groups.
\newblock {\em International Journal of Algebra and Computation}, 22(6), 2012.

\bibitem{BarSil}
L.~Bartholdi and P.~V. Silva.
\newblock {\em Rational subsets of groups}, volume Handbook of Automata Theory,
  chapter~23.

\bibitem{alf}
I.~Bondarenko, T.~Ceccherini-Silberstein, A.~Donno, and V.~Nekrashevych.
\newblock On a family of schreier graphs of intermediate growth associated with
  a self-similar group.
\newblock {\em European J. Combin.}, 33(7):1408--1421, 2012.

\bibitem{BDN}
I.~Bondarenko, D.~D'Angeli, and T.~Nagnibeda.
\newblock Ends of {S}chreier graphs of self-similar groups.
\newblock {\em In preparation}.

\bibitem{basilica}
D.~D'Angeli, A.~Donno, M.~Matter, and T.~Nagnibeda.
\newblock Infinite {S}chreier graphs of the {B}asilica group.
\newblock {\em Journal of Modern Dynamics}, 2(24):153--194, 2010.

\bibitem{DaRo2013}
D.~D'Angeli and E.~Rodaro.
\newblock Groups and semigroups defined by colorings of synchronizing automata.
\newblock 2013.

\bibitem{Eil}
S.~Eilenberg.
\newblock {\em Automata, Languages, and Machines}, volume~A of {\em Pure and
  Applied Mathematics}.
\newblock Academic Press, 1974.

\bibitem{DynSubgroup}
R.~I. Grigorchuk.
\newblock Some topics of dynamics of group actions on rooted trees.
\newblock {\em The Proceedings of the Steklov Institute of Math.}, 273:1--118,
  2011.

\bibitem{GriSa13}
R.~I. Grigorchuk and D.~Savchuk.
\newblock Self-similar groups acting essentially freely on the boundary of the
  binary rooted tree.
\newblock {\em To appear in Contemporary Mathematics}, 2013.

\bibitem{HowieAuto}
J.~M. Howie.
\newblock {\em Automata and Languages}.
\newblock Clarendon Press, 1991.

\bibitem{KaSiSte}
M.~Kambite, P.~V. Silva, and B.~Steinberg.
\newblock The spectra of lamplighter groups and cayley machines.
\newblock {\em Geom. Dedicata}, 120:193--227, 2006.

\bibitem{KaMi02}
I.~Kapovich and A.~Myasnikov.
\newblock Stallings foldings and the subgroup structure of free groups.
\newblock {\em Journal of Algebra}, 248(2):608--668., 2002.

\bibitem{volo}
V.~Nekrashevych.
\newblock Self-similar groups.
\newblock {\em Mathematical Surveys and Monographs, American Mathematical
  Society, Providence, RI}, 117, 2005.

\bibitem{Serre}
J.P. Serre.
\newblock {\em Trees}.
\newblock Springer-Verlag, 1980.

\bibitem{SilvaVirtual}
P.~V. Silva.
\newblock Fixed points of endomorphisms of virtually free groups.
\newblock {\em Pacific J. Math.}, 263(1):207--240, 2013.

\bibitem{SiSte}
P.~V.. Silva and B.~Steinberg.
\newblock On a class of automata groups generalizing lamplighter groups.
\newblock {\em International Journal of Algebra and Computation},
  15(05n06):1213--1234, 2005.

\bibitem{StVoVo2011}
B.~Steinberg, M.~Vorobets, and Y.~Vorobets.
\newblock Automata over a binary alphabet generating free groups of even rank.
\newblock {\em Int. J. of Algebra and Comput.}, 21(12):329--354, 2011.

\bibitem{Vershik}
A.~M. Vershik.
\newblock Totally nonfree actions and the infinite symmetric group.
\newblock {\em Mosc. Math. J.}, 12(1):193--212, 2012.

\bibitem{VoVo2007}
M.~Vorobets and Y.~Vorobets.
\newblock On a free group of transformations defined by an automaton.
\newblock {\em Geom. Dedicata}, 124:237--249, 2007.

\bibitem{VoVo2010}
M.~Vorobets and Y.~Vorobets.
\newblock On a series of finite automata defining free transformation groups.
\newblock {\em Groups, Geom. and Dynamics}, 4(2):377--405, 2010.

\bibitem{Zelma90}
E.~I. Zelmanov.
\newblock Solution of the restricted burnside problem for groups of odd
  exponent.
\newblock {\em Izv. Akad. Nauk SSSR Ser. Mat.}, 54(1):42--59, 1990.

\bibitem{Zelma91}
E.~I. Zelmanov.
\newblock Solution of the restricted burnside problem for 2-groups.
\newblock {\em Mat. Sb.}, 182(4):568--592, 1991.

\end{thebibliography}

\end{document}